\documentclass[leqno]{amsart}
\pdfoutput=1
\usepackage[a4paper, margin=3cm]{geometry}
\usepackage{amsthm}
\usepackage{amsmath,amssymb,mathtools,mathdots}
\usepackage{mathrsfs}
\usepackage{url}
\usepackage[all]{xy}
\usepackage{tikz}
\usepackage{tikz-cd}
\usetikzlibrary {arrows.meta,backgrounds}
\usetikzlibrary{shapes.geometric, intersections, calc}
\usetikzlibrary{decorations.pathmorphing, arrows.meta}
\usepackage{comment}
\usepackage{placeins}

\newtheorem{theorem}{Theorem}[section]
\newtheorem*{theorem*}{Theorem}
\newtheorem{lemma}[theorem]{Lemma}
\newtheorem*{lemma*}{Lemma}
\newtheorem{proposition}[theorem]{Proposition}
\newtheorem*{proposition*}{Proposition}
\newtheorem{corollary}[theorem]{Corollary}
\newtheorem*{corollary*}{Corollary}

\newtheorem*{claim*}{Claim}

\newtheorem*{fact*}{Fact}

\newtheorem*{conjecture*}{Conjecture}
\theoremstyle{definition}
\newtheorem{definition}[theorem]{Definition}
\newtheorem*{definition*}{Definition}
\newtheorem{example}[theorem]{Example}
\newtheorem*{example*}{Example}
\newtheorem{remark}[theorem]{Remark}
\newtheorem*{remark*}{Remark}

\newtheorem*{question*}{Question}

\newtheorem*{assumption*}{Assumption}
\numberwithin{equation}{section}
\allowdisplaybreaks[1]

\newcommand{\abs}[1]{\left\lvert#1\right\rvert}

\DeclareMathOperator{\End}{End}
\DeclareMathOperator{\rank}{rank}

\DeclareMathOperator{\Ad}{Ad}
\DeclareMathOperator{\ad}{ad}

\DeclareMathOperator{\tr}{tr}

\DeclareMathOperator{\id}{id}
\DeclareMathOperator{\im}{im}

\DeclareMathOperator{\GL}{GL}
\DeclareMathOperator{\Mat}{Mat}
\DeclareMathOperator{\SL}{SL}
\DeclareMathOperator{\Gal}{Gal}

\DeclareMathOperator{\Aut}{Aut}
\DeclareMathOperator{\pr}{pr}

\let\Re\relax
\DeclareMathOperator{\Re}{Re}
\let\Im\relax
\DeclareMathOperator{\Im}{Im}
\newcommand{\C}{{\mathbb C}}
\newcommand{\R}{{\mathbb R}}
\newcommand{\Q}{{\mathbb Q}}
\newcommand{\Z}{{\mathbb Z}}

\newcommand{\Half}{{\mathbb H}}
\newcommand{\w}{{\, \wedge \,}}
\newcommand{\pa}{{\partial}}
\newcommand{\bpa}{{\overline{\partial}}}

\newcommand{\conj}[1]{{\overline{#1}}}
\newcommand{\met}[2]{{\langle{#1},{#2}\rangle}}

\newcommand{\fa}{{\mathfrak a}}

\newcommand{\fd}{{\mathfrak d}}

\newcommand{\fg}{{\mathfrak g}}
\newcommand{\fh}{{\mathfrak h}}

\newcommand{\fk}{{\mathfrak k}}

\newcommand{\fm}{{\mathfrak m}}
\newcommand{\fn}{{\mathfrak n}}

\newcommand{\fu}{{\mathfrak u}}

\newcommand{\OO}[1]{{\mathcal O_{#1}}}
\newcommand{\unit}[1]{{{\mathcal O}_{#1}^{\times}}}
\newcommand{\unitp}[1]{{{\mathcal O}_{#1}^{\times, +}}}
\usepackage{multirow}

\address{graduate school of mathematical sciences, the university of tokyo, 3-8-1 komaba, meguro, tokyo 153-8914, japan.}
\email{shuho@ms.u-tokyo.ac.jp}
\begin{document}
\title
[
A characterization of OT manifolds 
in LCK geometry
]
{
A characterization of Oeljeklaus-Toma manifolds 
in locally conformally K\"{a}hler geometry
}
\author{Shuho Kanda}
\date{}
\maketitle
\begin{abstract}
    We show that for a certain class of solvable Lie groups, 
if they admit a left-invariant non-Vaisman 
locally conformally K\"{a}hler metric and a lattice, 
they must arise from the construction of 
Oeljeklaus-Toma manifolds. 
This result provides a natural explanation for 
why number-theoretic considerations play a role in the construction of 
Oeljeklaus-Toma manifolds. 

\end{abstract}

\setcounter{tocdepth}{1}
\tableofcontents

\section{Introduction}

\subsection{Overview}

Oeljeklaus-Toma (OT, for short) manifolds, 
introduced by Oeljeklaus and Toma in \cite{OT05}, 
form a class of compact non-K\"{a}hler manifolds. 
They have solvmanifold structures \cite{Kas13b}, 
and a large subclass of them 
admits locally conformally K\"{a}hler (LCK, for short) metrics 
which are not Vaisman. 
As a result, 
OT manifolds have become an important class in the study of LCK geometry.
The construction of OT manifolds relies on number theory. 
In general, 
constructing lattices in solvable Lie groups is notoriously difficult. 
However, 
by using number fields, 
their rings of integers, 
and unit groups, 
one can obtain solvmanifolds with left-invariant complex structures. 
Interestingly, 
no other solvmanifolds are known to admit non-Vaisman LCK metrics 
aside from OT manifolds. 
This raises two fundamental questions: 
\emph{Why have no other examples been found?} 
And \emph{why does number theory play a crucial role in the only known case?}

In this paper, 
we prove that if a certain class of solvmanifolds 
admits a left-invariant LCK metric, 
then they must arise from the construction of OT manifolds. 
This result provides a natural explanation for 
why number-theoretic considerations play a role 
in the construction of lattices 
and suggests that OT manifolds provide a natural example 
in the study of LCK solvmanifolds.

\subsection{Background}

\subsubsection{Locally conformally K\"{a}hler geometry} 

A Hermitian manifold $(M,J,g)$ is called LCK 
if there exists an open covering $\{U_i\}_{i \in I}$ of $M$ 
and real-valued functions $f_i \in C^{\infty}(U_i)$ 
such that $e^{-f_i}g|_{U_i}$ are K\"{a}hler. 
The metric $g$ is LCK if and only if 
there exists a closed $1$-form $\theta$, 
which is called the \emph{Lee form}, 
such that $d\omega = \theta \w \omega$, 
where $\omega(\cdot \, , \cdot) \coloneqq g(J\cdot \, ,\cdot)$ 
is the fundamental form. 
These manifolds are generalizations of K\"{a}hler manifolds, 
and they have been much studied by many authors since the work of I. Vaisman in \cite{Vai76}. 
See for example \cite{DO98}, \cite{OV24}. 

These manifolds admit a special type of LCK metric known as the \emph{Vaisman metric}, 
which is considered closer to being K\"{a}hler.  
A Vaisman metric is an LCK metric for which the Lee form is parallel, 
that is, 
$\nabla \theta = 0$. 
Compact complex manifolds that do not admit any K\"{a}hler metrics 
but admit Vaisman metrics include classical examples such as Hopf manifolds 
and Kodaira-Thurston manifolds, 
both of which have long been known as non-K\"{a}hler manifolds. 

Vaisman metrics have been extensively studied, 
but they exhibit exceptionally good properties within the class of LCK manifolds. 
As a result, 
various invariants defined on LCK manifolds 
often become trivial when restricted to Vaisman manifolds. 
This motivates the study of non-Vaisman LCK metrics. 
The OT manifolds, 
which will be defined later, 
are significant as they provide 
a family of manifolds admitting a non-Vaisman LCK metric. 

\subsubsection{Solvmanifolds}

A classical method for obtaining compact non-K\"{a}hler manifolds 
is to take a simply connected solvable Lie group $G$ 
and quotient it by a lattice $\Gamma$, 
a discrete co-compact subgroup of $G$, 
forming the manifold $\Gamma \backslash G$, 
which is called a \emph{solvmanifold}. 
A left-invariant LCK structure on $G$, 
or equivalently, 
an LCK structure on the Lie algebra $\fg$ of $G$, 
gives rise naturally to 
an LCK structure on $\Gamma \backslash G$. 
Solvmanifolds are valuable as examples since 
various invariants can be explicitly computed on them 
(see for example \cite{Kas13a}). 

Solvmanifolds admitting a lattice are unimodular. 
So to obtain a solvmanifold admitting a non-Vaisman LCK metric, 
it is sufficient to construct 
\begin{itemize}
    \item[(i)] a unimodular solvable Lie algebra $\fg$ 
    admitting a non-Vaisman LCK metric, and 
    \item[(ii)] a lattice in the Lie group $G$ associated to $\fg$. 
\end{itemize}
Both (i) and (ii) are notoriously difficult to construct. 
In particular, 
for (ii), 
determining whether a solvable Lie group admits a lattice 
and explicitly constructing one 
is a highly non-trivial problem in general (see Example \ref{Ex:Sol3}). 
It is easy to determine whether a nilpotent Lie group admits a lattice 
as shown in \cite{Mal49}. 
However, 
in \cite{Saw07}, 
it is shown that any metric on unimodular nilpotent Lie algebras is Vaisman. 
In \cite{Saw21} and \cite{Saw24}, 
the author studied (i). 
In \cite{AO18}, 
the authors studied (i) and (ii) for Lie groups of the form 
$G=\R \ltimes \R^n$, 
and concluded that $n=3$ is necessary for $G$ admits a 
left-invariant non-Vaisman LCK metric and a lattice. 

In the construction of OT manifolds, 
number-theoretic considerations, 
such as Dirichlet’s unit theorem, 
are used to effectively construct (i) and (ii)
simultaneously. 

\subsubsection{Oeljeklaus-Toma manifolds} 
Let $K$ be a number field of degree $n = [K : \Q]$ 
which admits $s$ real embeddings 
$\sigma_1, \ldots ,\sigma_s : K \hookrightarrow \R$, 
and $2t$ complex embeddings 
$\sigma_{s+1}, \ldots ,\sigma_{s+2t} : K \hookrightarrow \C$ 
where $\sigma_{s+i} = \conj{\sigma_{s+t+i}}$. 
We assume that $s,t \ge 1$. 
Taking a free subgroup $U \subset \unit{K}$ of rank $s$
that satisfies appropriate conditions (Definition \ref{Def:admissible}), 
we can define a compact complex manifold $X(K,U)$ of 
complex dimension $s+t$. 
We call this manifold an \emph{OT manifold} of type $(s,t)$. 

Since their construction in \cite{OT05}, 
OT manifolds have been the subject of extensive research. 
In the context of LCK geometry, 
it is shown that an OT manifold admits an LCK metric 
if and only if its type is $(s,1)$ 
(\cite{OT05}, \cite{Vul14}, \cite{Dub14}, \cite{DV23}). 
In \cite{Kas13b}, 
it is proved that no OT manifold admits a Vaisman metric 
and that OT manifolds admit solvmanifold structures. 
We here just describe the solvmanifold structures. 

\begin{definition}
    Let $C=(c_{ij})_{ij} \in \Mat_{t\times s}(\C)$ be a complex matrix 
    such that
    \[
    \Re c_{1j}+\cdots+\Re c_{tj}=-\frac{1}{2}
    \]
    for all $1\le j \le s$. 
    We define a unimodular solvable Lie group
    $G_C=\R^s \ltimes_{\phi_C}(\R^s \oplus\C^t)$ 
    where the map $\phi_C : \R^s \to \GL(\R^s \oplus \C^t)$ 
    is the following: 
    \[
    \phi_C(t_1, \ldots ,t_s) = 
    \exp \left(
    \mathrm{diag} (t_1, \ldots ,t_s,(\sum_{j=1}^s c_{ij} t_j)_{i=1}^t)
    \right). 
    \]
    We call a Lie group \emph{OT-like} of type $(s,t)$ if it is 
    isomorphic to $G_C$ 
    for some matrix $C \in \Mat_{t\times s}(\C)$. 
\end{definition}

\begin{theorem}[\cite{Kas13b}]
    Let $X(K,U)$ be an OT manifold of type $(s,t)$. 
    Take the basis $(a_i)_{i=1}^s$ of $U$ and 
    a matrix $C=(c_{ij})_{ij} \in \Mat_{t \times s}(\C)$ so that 
    \[
    \sigma_{s+i}(a_k)=
    \exp \left( \sum_{j=1}^s c_{ij} \log(\sigma_j(a_k)) \right)
    \]
    for all $1 \le i \le t$. 
    Then, 
    there is a lattice $\Gamma$ in $G_C$ such that 
    $X(K,U) \simeq \Gamma \backslash G_C$. 
\end{theorem}

In this case, 
we say the OT-like Lie group $G_C$ is 
associated to the OT manifold $X(K,U)$. 
The construction of OT manifolds can be viewed as a method for constructing 
a matrix $C$ such that the OT-like Lie group $G_C$ admits a lattice. 
It is shown that an OT-like Lie group $G_C$ admits 
a left-invariant LCK structure 
if and only if the matrix $C$ satisfies $\Re{c_{ij}}=-1/2t$ 
for all $i,j$. 
We call $G_C$ \emph{LCK OT-like} if this condition is satisfied. 
Note that this condition is always satisfied when the type is $(s,1)$. 

Within this framework, the following natural question arise: 
\emph{Are there OT-like Lie groups admitting a lattice 
that do not come from the OT manifold construction?}. 

\subsection{Results}
Let $G_C$ be an OT-like Lie group. 
We call a lattice $\Gamma$ in $G_C=\R^s \ltimes_{\phi_C}(\R^s \oplus\C^t)$ \emph{of simple type} if 
the representation of $\Gamma$ on the $\Q$-vector space 
$(\Gamma \cap (\R^s \oplus\C^t)) \otimes_{\Z} \Q$ is irreducible. 
We call an OT manifold $X(K,U)$ \emph{of simple type} if $\Q(U)=K$. 
We can easily show that an OT-like Lie group $G_C$ associated to $X(K,U)$ admits 
a lattice of simple type if $X(K,U)$ is of simple type. 

\begin{theorem}
[= Theorem \ref{Thm:OT-like Lie group with simple lattice is constructed by OT}]
\label{Thm:Main}
    Let $G$ be an OT-like Lie group of type $(s,t)$. 
    If $G$ admits a lattice of simple type, 
    $G$ is associated to some OT manifold $X(K,U)$ 
    of type $(s,t)$ and of simple type. 
\end{theorem}

Fortunately, 
it turns out that the assumption of lattice simplicity 
is not needed for LCK OT-like Lie groups
(Proposition \ref{Prop:LCK OT-like is simple}). 
So we have the following 

\begin{theorem}
[= Corollary \ref{Cor:LCK OT-like with lattice is OT}]
\label{Thm:Main2}
    Let $G$ be an LCK OT-like Lie group of type $(s,t)$. 
    If $G$ admits a lattice, 
    we have $t=1$. 
    Moreover, 
    $G$ is associated to some OT manifold $X(K,U)$ 
    of type $(s,1)$ and of simple type. 
\end{theorem} 

A priori, 
the relationships among OT-like Lie groups that 
admitting a lattice (of simple type), 
admitting a left-invariant LCK metric, 
and those associated to OT manifolds can be summarized as shown in the left diagram 
in Figure \ref{Fig:OT-like}. 
The meanings of the terms in Figure \ref{Fig:OT-like} are as follows: 
\begin{itemize}
    \item ``(LCK) OT-like'' denotes (LCK) OT-like Lie groups. 
    \item ``(simple) lattice'' denotes OT-like Lie groups admitting 
    a lattice (of simple type). 
    \item ``(simple) OT'' denotes Lie groups associated to some OT manifolds 
    (of simple type). 
    \item ``LCK OT'' denotes Lie groups associated to some OT manifolds 
    admitting an LCK metric. 
\end{itemize}
With this notation, 
Theorem \ref{Thm:Main} says that 
$[\text{simple OT}]=[\text{simple lattice}]$, 
and Theorem \ref{Thm:Main2} says that 
$[\text{LCK OT-like}] \cap [\text{lattice}] \subset [\text{simple OT}]$, 
which simplify the left diagram in Figure \ref{Fig:OT-like} into 
the right diagram. 

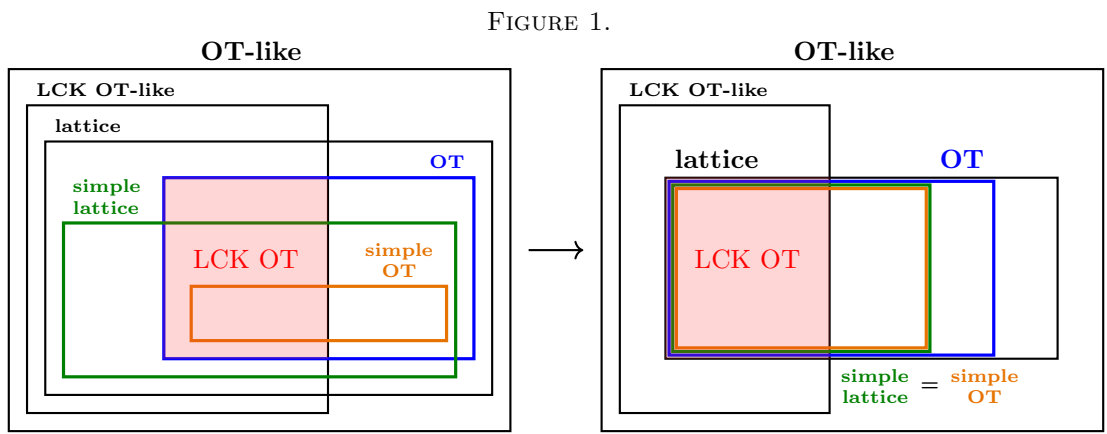
\begin{figure}[h]
\caption{}
\label{Fig:OT-like}
\centering
\begin{tikzpicture}[scale=1.2]

\draw[thick] (-6, -2) rectangle (-0.5, 2);
\node[anchor=south west] at (-4, 2) {\textbf{OT-like}};

\draw[thick] (-5.8, -1.8) rectangle (-2.5, 1.6);
\node[anchor=south west] at (-5.8, 1.6) {\scriptsize{\textbf{LCK OT-like}}};

\draw[thick] (-5.6, -1.6) rectangle (-0.7, 1.2);
\node[anchor=south west] at (-5.6, 1.2) {\scriptsize{\textbf{lattice}}};

\draw[very thick,draw=blue] (-4.3, -1.2) rectangle (-0.9, 0.8);
\node[anchor=south east,text=blue] at (-0.9, 0.8) {\scriptsize{\textbf{OT}}};

\draw[very thick,draw=green!50!black] (-5.4, -1.4) rectangle (-1.1, 0.3);
\node[anchor=south west,text=green!50!black] at (-5.4, 0.3) 
{\scriptsize \parbox{0.1cm} {\centering \textbf{simple\\lattice}}};

\draw[very thick,draw=orange!90!black] (-4, -1) rectangle (-1.2, -0.4);
\node[anchor=south east,text=orange!90!black] at (-1.2, -0.4) 
{\scriptsize \parbox{1cm} {\centering \textbf{simple\\OT}}};

\fill[red!80, opacity=0.2] (-4.3, -1.2) rectangle (-2.5, 0.8);
\node[text=red] at (-3.4,-0.1) {LCK OT};


\draw[thick] (6, -2) rectangle (0.5, 2);
\node[anchor=south west] at (2.5, 2) {\textbf{OT-like}};

\draw[thick] (0.7, -1.8) rectangle (3, 1.6);
\node[anchor=south west] at (0.7, 1.6) {\scriptsize{\textbf{LCK OT-like}}};

\draw[thick] (1.2, -1.2) rectangle (5.5, 0.8);
\node[anchor=south west] at (1.2, 0.8) {\textbf{lattice}};

\draw[very thick,draw=blue] (1.24, -1.16) rectangle (4.8, 0.76);
\node[anchor=south east,text=blue] at (4.8, 0.8) {\textbf{OT}};

\draw[very thick,draw=green!50!black] (1.28, -1.12) rectangle (4.1, 0.72);

\draw[very thick,draw=orange!90!black] (1.32, -1.08) rectangle (4.06, 0.68);

\node[text=green!50!black] at (3.5,-1.5) 
{\scriptsize \parbox{1cm} {\centering \textbf{simple\\lattice}}};

\node[text=orange!90!black] at (4.7,-1.5) 
{\scriptsize \parbox{1cm} {\centering \textbf{simple\\OT}}};

\node at (4.1,-1.5) {\scriptsize{\textbf{=}}};

\fill[red!80, opacity=0.2] (1.2, -1.2) rectangle (3, 0.8);

\node[text=red] at (2.1,-0.1) {LCK OT};

\draw[thick, ->] 
    (-0.3, 0) -- (0.3, 0) node[midway, above] {};

\end{tikzpicture}
\end{figure}

In this paper, 
we also study how LCK OT-like Lie algebras, 
Lie algebras associated to LCK OT-like Lie groups, 
relate to the 
classification of LCK Lie algebras. 
We call a Lie algebra $\fg$ \emph{meta-abelian} if $\fg$ is isomorphic to 
the semi-direct product of two abelian Lie algebras $\R^m \ltimes_{d\phi} \R^n$. 
It is known that meta-abelian Lie algebras admitting a Vaisman metric are 
isomorphic to $\R \times \fh_{2d+1}$, 
where $\fh_{2d+1}$ denotes the Heisenberg Lie algebra 
of dimension $2d+1$ \cite{Saw07}. 
For the non-Vaisman metrics, 
we show the following 

\begin{theorem}
[= Theorem \ref{Thm:classification of LCK meta-abelian}]
\label{Thm:Main3}
    Let $\fg=\R^m\ltimes_{d\phi}\R^n$ be a unimodular 
    meta-abelian Lie algebra and 
    $(\fg, J, \met{\cdot}{\cdot})$ be a non-Vaisman LCK structure. 
    If $\R^n+J\R^n=\fg$, 
    then $\fg$ is an LCK OT-like Lie algebra of type $(s,t)$, 
    where $m=s, n=s+2t$. 
\end{theorem} 

The assumption $\R^n+J\R^n=\fg$ is not needed 
when considering the case $m=1$ or $m=2$ 
(Lemma \ref{Lem:assumption is not needed when m=1,2}). 
By combining the results so far, 
we have the following 

\begin{theorem}
[= Theorem \ref{Thm:classification when m=1,2}]
    Let $G=\R^m\ltimes_{\phi}\R^n$ be a unimodular 
    meta-abelian Lie group with $m=1$ or $m=2$. 
    If $G$ admits a left-invariant non-Vaisman LCK metric  
    and a lattice, 
    we have $n=m+2$ and 
    $G$ is associated to some OT manifold of type $(m,1)$. 
\end{theorem}

\subsection{Strategy of the proof}

The proof of Theorem \ref{Thm:Main} is likely to be the most non-trivial, 
as it involves constructing a number-theoretic object 
from a Lie group and its lattice. 
Here, 
we briefly outline the idea of the proof of Theorem \ref{Thm:Main}. 

Let $G_C=\R^s \ltimes_{\phi_C} (\R^s \oplus \C^t)$ 
be an OT-like Lie group admitting a lattice of simple type. 
From Proposition \ref{Prop:meta-abelian with a lattice admits splitting a lattice}, 
whose proof essentially relies on Mostow's theorem \cite{Mos54}, 
we can show that $G_C$ admits a lattice of the form
$\Gamma=\Gamma_1 \ltimes_{\phi_C} (P^{-1} \Z^n)$, 
where $n=s+2t$, $P \in \GL(n,\R)$ and $\Gamma_1$ is a lattice in $\R^s$. 
The lattice may be different from the original one, 
but it is of simple type. 
$P$ satisfies the following condition: 
\[
P (\phi_C(x))P^{-1} \in \SL(n,\Z), \quad \text{for all} \;x \in \Gamma_1. 
\]
We define a subgroup of $\SL(n,\Z)$ as follows: 
\[
U \coloneqq \left\{
P (\phi_C(x))P^{-1} \in \SL(n,\Z)\, \mid \, 
x \in \Gamma_1 \right\}. 
\]
The group $U \subset \SL(n,\Z)$ is a free abelian subgroup of rank $s$. 
Then we can define a $\Q$-algebra 
$K \coloneqq \Q[U] \subset \Mat(n,\Q)$. 
By using the fact that the elements of $K$ are diagonalized by $P$ and 
the simplicity of the lattice, 
it is shown that $K$ is a field such that $[K : \Q]=n$ 
and that $U$ is a subgroup of $\unit{K}$. 
Eventually, 
we can show that $G_C$ is associated to an OT manifold $X(K,U)$. 

\subsection{Organization of this paper}
The paper is organized as follows: 
\begin{itemize}
\item In Section \ref{section:Preliminaries}, 
we review definitions and results on solvmanifolds and LCK manifolds.
\item In Section \ref{section:Oeljeklaus-Toma manifolds}, 
we review the construction of OT manifolds and their solvmanifold structures. 
We also summarize some results on LCK metrics of OT manifolds. 
\item In Section \ref{section:Construction of OT manifolds 
from Lie groups and their lattices}, 
we prove Theorem \ref{Thm:Main} and Theorem \ref{Thm:Main2}. 
\item In Section \ref{section:meta-abelian},  
we study LCK metrics on meta-abelian Lie algebras and 
prove Theorem \ref{Thm:Main3}. 
\item In Appendix \ref{Appendix}, 
we prove that an OT manifold admits an LCK metric if and only if 
it is of type $(s,1)$. 
We also present results on pluriclosed metrics. 
\end{itemize}

\addtocontents{toc}{\protect\setcounter{tocdepth}{0}}
\section*{Acknowledgments}
\addtocontents{toc}{\protect\setcounter{tocdepth}{1}}
I would like to express my gratitude to my supervisor, 
Kengo Hirachi, for his enormous support and helpful advice. 
I am also grateful to Hisashi Kasuya for helpful discussions. 
I extend my sincere thanks to Junnosuke Koizumi, 
who provided the core idea for Theorem 
\ref{Thm:OT-like Lie group with simple lattice 
is constructed by OT} and theorems in Appendix \ref{Appendix}. 
I had valuable discussions about Theorem 
\ref{Thm:classification of LCK meta-abelian} with Yugo Takanashi, 
and I am thankful for his contributions. 
Lastly, 
I would like to express my appreciation to Hiroshi Sawai 
for proposing the challenging and inspiring problem of 
classifying meta-abelian LCK structures, 
which motivated this research. 
This research is supported by FoPM, WINGS Program, the University of Tokyo, 
and JSPS KAKENHI Grant number 24KJ0931. 

\section{Preliminaries}
\label{section:Preliminaries}

\subsection{Solvmanifolds and lattices}

Let $G$ be a connected solvable Lie group. 
We call its discrete subgroup $\Gamma \subset G$ a \emph{lattice} if 
the quotient $\Gamma \backslash G$ is compact. 
In this case, 
we call the quotient manifold $\Gamma \backslash G$ a  \emph{solvmanifold}. 
We say a solvable Lie algebra $\fg$ admits a lattice 
if the simply connected Lie group associated to $\fg$ admits a lattice. 

\begin{remark}
    For a general Lie groups $G$, 
    the term ``lattice'' is used to refer to a discrete subgroup $\Gamma$ 
    such that the volume of the quotient space $\Gamma \backslash G$ is finite, 
    relative to its Haar measure. 
    However, 
    when considering solvable Lie groups, 
    this property is equivalent to the compactness of $\Gamma \backslash G$ \cite{OV00}. 
\end{remark}

\begin{remark}
    Let $M$ be a compact homogeneous space of a connected solvable Lie group. 
    Then there exists a solvmanifold $\Gamma \backslash G$ which is a 
    finite covering of $M$ \cite{Aus73}\cite{OV90}.  
\end{remark}

The easiest solvable Lie groups are abelian Lie groups $\R^n$. 
As for their lattices, 
the following lemma is easy but important: 

\begin{lemma}
    A subgroup $\Gamma \subset \R^n$ is a lattice in $\R^n$ 
    if and only if 
    there exists an $\R$-basis 
    $\{ \mathbf{v}_i \}_{i=1}^n$ in $\R^n$  
    such that $\Gamma = \Z \mathbf{v}_1 + \cdots + \Z \mathbf{v}_n$. 
\end{lemma}

In general, 
it is difficult to check whether a solvable Lie group admits 
a lattice or not. 
The fundamental necessary condition is the following

\begin{lemma}\label{Lem:admitting lattice then unimodular}
    Let $G$ be a connected solvable Lie group. 
    If $G$ admits a lattice, 
    then $G$ is unimodular, 
    that is, 
    for all $x \in G$ we have 
    \[
    \det (\Ad_x) = 1. 
    \]
    Equivalently, 
    $\mathfrak{g}$ is unimodular, 
    that is, 
    for all $X \in \mathfrak{g}$ we have 
    \[
    \tr(\ad_X)=0. 
    \]
\end{lemma}

In this article we focus on a certain family of solvable Lie groups, 
namely, 
the meta-abelian ones. 
We call a connected and simply connected Lie group $G$ \emph{meta-abelian} if 
$G$ is isomorphic to the semi-direct product of two abelian Lie groups. 
More precisely, 
such Lie groups are expressed as follows. 
Let $\R^m$ and $\R^n$ be abelian Lie groups and 
$\phi \colon \R^m \to \GL(\R^n)$ be a Lie group homomorphism. 
Then the Lie group $\R^m \ltimes_{\phi} \R^n$ is a manifold $\R^m \times \R^n$ 
equipped with the following product: 
\[
(x_1,y_1) \cdot (x_2,y_2) \coloneqq (x_1+x_2, y_1 + \phi(x_1)y_2). 
\]
Its Lie algebra, 
which is also called meta-abelian, 
is a semi-direct product of two abelian Lie algebras 
$\R^m \ltimes_{d\phi} \R^n$, 
where $d \phi \colon \R^m \to \End(\R^n)$ is the derivative of $\phi$. 
Its bracket is expressed as follows: 
\[
[(X_1,Y_1),(X_2,Y_2)] = (0,d\phi(X_1)Y_2 - d\phi(X_2)Y_1). 
\]

\begin{example}\label{Ex:Sol3}
    Let $G=\R \ltimes_{\phi} \R^2$ be a $3$-dimensional Lie group, 
    where the map $\phi \colon \R \to \GL(\R^2)$ is defined as follows: 
    \[
    \phi(x) \coloneqq 
    \begin{pmatrix}
        e^x & 0 \\
        0 & e^{-x} 
    \end{pmatrix}. 
    \]
    The Lie algebra $\mathfrak{g}$ of $G$ is $\R \ltimes_{d\phi}\R^2$ 
    ,where the map $d\phi \colon \R \to \End(\R^2)$ is as follows: 
    \[
    d\phi(X) \coloneqq 
    \begin{pmatrix}
        X & 0 \\
        0 & -X 
    \end{pmatrix}. 
    \]
    Notice that $G$ is unimodular. 
    We shall construct a lattice in $G$
    (cf. \cite{SY05}, \cite{Saw07lattice}). 
    Let $M \in \SL(2,\Z)$ be a matrix whose eigenvalues are $\lambda$ and $\lambda^{-1}$ such that 
    $0<\lambda^{-1}<\lambda$. 
    Let $P \in \GL(2,\R)$ be its diagonalization, 
    that is, 
    $P^{-1}MP=\text{diag}(\lambda,\lambda^{-1})$. 
    Then we obtain a lattice $\Gamma$ in G as follows: 
    \[
    \Gamma \coloneqq ((\log\lambda) \Z) \ltimes_{\phi} (P^{-1}\Z^2).  
    \]
\end{example}

For a certain class of meta-abelian Lie groups, 
we may obtain an explicit description of lattices as in the above example. 
We call the largest connected nilpotent normal subgroup in $G$ the \emph{nilradical} and 
its Lie algebra the \emph{nilradical} of $\mathfrak{g}$. 
The nilradical of $\mathfrak{g}$ corresponds to the largest nilpotent ideal of $\mathfrak{g}$. 
We consider 
the case where the nilradical of a meta-abelian Lie group $G=\R^m \ltimes_{\phi} \R^n$ 
coincides with $\R^n$. 
Equivalently, 
the nilradical of its Lie algebra $\R^m \ltimes_{d\phi} \R^n$ 
coincides with $\R^n$. 
In this case, 
if $G$ admits a lattice, 
we can construct a lattice in $G$ using the method as in Example \ref{Ex:Sol3} (cf. \cite[Lemma 2.3]{TY01})
as shown in the following

\begin{proposition}
\label{Prop:meta-abelian with a lattice admits splitting a lattice}

    Let $G=\R^m \ltimes_{\phi} \R^n$ be a meta-abelian Lie group. 
    If $G$ admits a lattice $\Gamma$ and the nilradical of $G$ coincides with $\R^n$, 
    then there exists a lattice $\Gamma_1$ in $\R^m$ and a matrix $P \in \GL(n,\R)$ 
    such that for all $x \in \Gamma_1$ we have
    \[
    P \phi(x)P^{-1} \in \SL(n,\Z). 
    \]
    In this case, 
    $\Gamma' \coloneqq \Gamma_1 \ltimes_{\phi} (P^{-1}\Z^n)$ 
    is a lattice in $G$. 
\end{proposition}

\begin{proof}
    By Mostow's theorem \cite{Mos54} \cite{Rag72} \cite{Aus73} , 
    the intersection of the lattice $\Gamma$ and the nilradical 
    is a lattice in the nilradical. 
    Thus $\Gamma_0 \coloneqq \Gamma \cap \R^n$ is a lattice in $\R^n$. 
    We take a $\Z$-basis $\{ \textbf{v}_i \}_{i=1}^n$ of 
    the lattice $\Gamma_0$ in $\R^n$ and 
    a matrix $P \in \GL(n,\R)$ so that $P \mathbf{v}_i=\mathbf{e}_i$ 
    for all $i$, 
    where $\mathbf{e}_i$ denote standard basis vectors. 

    We define a subgroup $\Gamma_1 \subset \R^m$ as follows:  
    \[
    \Gamma_1 \coloneqq \{ x \in \R^m \mid 
    \exists y \in \R^n \text{ s.t. } (x,y) \in \Gamma \}. 
    \]
    We will show that $\Gamma_1$ is a lattice in $\R^m$. 
    Assume by contradiction that there exists a point $x_{\infty} \in \Gamma_1$ and 
    a sequence $\{ x_i \}_{i=1}^{\infty} \subset \Gamma_1$ such that 
    $x_i \to x_{\infty}$ $(i \to \infty)$ 
    and $x_i \ne x_{\infty}$ for all $i$. 
    Since $\Gamma_0$ is a lattice in $\R^n$, 
    there is a compact set $K \subset \R^n$ such that for all $x_i \in \Gamma_1$ 
    we can take $y_i \in K$ so that $(x_i,y_i) \in \Gamma$. 
    As $K$ is compact, 
    there is a subsequence $\{ y_{s(i)} \}_{i=1}^{\infty}$ such that 
    $y_{s(i)}$ converges to some point in $\R^n$. 
    Therefore the sequence $(x_{s(i)},y_{s(i)})$ converges to some point in $G$. 
    This contradicts the discreteness of $\Gamma$. 
    There is a natural continuous surjection  
    $\Gamma \backslash G \to \Gamma_1 \backslash \R^m$. 
    Since $\Gamma \backslash G$ is compact, 
    so is $\Gamma_1 \backslash \R^m$. 
    Therefore, $\Gamma_1$ is a lattice in $\R^m$. 

    For all $x \in \Gamma_1$ and $y \in \Gamma_0$, 
    take $y' \in \R^n$ so that $(x,y') \in \Gamma$. 
    Since we have $(x,y') \cdot (0,y) \cdot (x,y')^{-1}=(0,\phi(x)y) 
    \in \Gamma_0$, 
    it follows that $\phi|_{\Gamma_1}$ acts on $\Gamma_0$. 
    By the definition of $P$, 
    $\Gamma' \coloneqq \Gamma_1 \ltimes_{\phi} (P^{-1}\Z^n)$ 
    is a discrete subgroup in $G$. 
    Moreover, 
    we have $P \phi(x)P^{-1} \in \GL(n,\Z)$ for all $x \in \Gamma_1$. 
    In fact, 
    we have 
    $P \phi(x)P^{-1} \in \SL(n,\Z)$ by the unimodularity. 
    It is easy to see that $\Gamma' \backslash G$ is compact.
\end{proof}

Lastly, 
we define lattices of simple type. 
Let $G=\R^m \ltimes_{\phi} \R^n$ be a meta-abelian Lie group 
whose nilradical coincides with $\R^n$. 
Define a lattice $\Gamma_0 \coloneqq \Gamma \cap \R^n$ 
as in the above proof. 
The adjoint representation 
$\Ad|_{\Gamma} \colon \Gamma \to \Aut(\Gamma_0)$
defines a representation of $\Gamma$ on the $\Q$-vector space 
$\Gamma_0 \otimes_{\Z} \Q$. 

\begin{definition}\label{Def:lattice of simple type}
    Let $G=\R^m \ltimes_{\phi} \R^n$ be a meta-abelian Lie group 
    whose nilradical coincides with $\R^n$. 
    We call a lattice $\Gamma$ in $G$ \emph{of simple type} if 
    the representation of $\Gamma$ on the $\Q$-vector space 
    $\Gamma_0 \otimes_{\Z} \Q$ is irreducible. 
\end{definition}

It is obvious that if a lattice $\Gamma$ is of simple type, 
then the lattice $\Gamma'$ in Proposition 
\ref{Prop:meta-abelian with a lattice admits splitting a lattice} is 
also of simple type. 

\subsection{Locally conformally K\"{a}hler structures on 
manifolds}

Let $(M,J)$ be a complex manifold with a Hermitian metric $g$. 
We consider the fundamental form 
$\omega(\cdot \, , \cdot) \coloneqq g(J\cdot \, ,\cdot)$. 
The metric $g$ is said to be 
\emph{conformally K\"{a}hler} (CK, for short) on $M$ 
if there exists a real-valued function $f$ 
such that $e^{-f} g$ is K\"{a}hler on $M$, 
i.e. $d(e^{-f} \omega)=0$. 
We call the metric $g$ LCK  
if there exists an open covering $\{U_i\}_{i \in I}$ of $M$ 
such that $g|_{U_i}$ are CK on $U_i$ for all $i \in I$. 
For an LCK metric $g$, 
take a real-valued function $f_i$ on $U_i$ 
so that $e^{-f_i}g$ is K\"{a}hler on $U_i$. 
Then, the condition that the form $e^{-f_i}g$ is K\"{a}hler is equivalent to 
$d\omega=df_i \w \omega$. 
Since $\omega$ is nondegenerate, 
$df_i \w \omega=df_j \w \omega$ implies 
$df_i=df_j$ on $U_i \cap U_j$. 
So the $1$-forms $\{df_i\}_{i \in I}$ are glued together. 
Thus we obtain the following proposition 
which makes the above definition simpler: 
\begin{proposition}
    A Hermitian metric $g$ on a complex manifold $(M,J)$ 
    is LCK if and only if 
    there exists a closed $1$-form $\theta$ on $M$ 
    such that $d\omega = \theta \w \omega$, 
    where $\omega$ is the fundamental form. 
    Furthermore, 
    the metric $g$ is CK if and only if $\theta$ is exact. 
\end{proposition} 

The $1$-form $\theta$ is determined by the metric $g$, 
and we call it the \emph{Lee form} of $g$. 
In this situation we call the Hermitian manifold 
$(M,J,g)$ an \emph{LCK manifold} with the Lee form $\theta$. 
Typical examples of LCK manifolds are Hopf manifolds. 
The construction of the LCK metrics shown in the following example 
uses the same method as the construction of the LCK metrics on 
Oeljeklaus-Toma manifolds, 
which will be presented later in Proposition
\ref{Prop:OT of type (s,1) admits LCK}. 

\begin{example}[Hopf manifolds]
    Let $\lambda \in \C$ be a complex constant such that $\abs{\lambda}>1$ 
    and $\Z$ be a cyclic group generated by $\tau_{\lambda}$. 
    We define an action of $\Z$ on a complex manifold 
    $X \coloneqq \C^n \backslash \{ 0 \}$ 
    by $\tau_{\lambda} \cdot z=\lambda z$. We call the quotient manifold 
    $X/\Z$ a \emph{Hopf manifold}. 
    Let $g$ be a K\"{a}hler metric on $X$ 
    defined as the restriction of the standard Hermitian metric on $\C^n$. 
    Define a function 
    $\Psi(z) \coloneqq \abs{z}^{-2} \colon X \to \R_{>0}$. 
    It is easy to check that $\Psi g$ is $\Z$-invariant, 
    and defines a metric on $X/\Z$. 
    By the construction, 
    this metric on $X/\Z$ is LCK. 
\end{example}

There is a special class of LCK manifolds known as Vaisman manifolds \cite{Vai82}. 
We call an LCK manifold $(M,J,g)$ \emph{Vaisman} if 
its Lee form $\theta$ is parallel with respect to 
the Levi-Civita connection $\nabla$, 
that is, 
$\nabla \theta=0$ holds. 
It is known that the 
Hopf manifolds with LCK metrics constructed in the above example is Vaisman. 

\subsection{Locally conformally K\"{a}hler structures on 
Lie algebras} 

Let $G$ be a connected Lie group and $\mathfrak{g}$ be 
the associated Lie algebra. 
We define LCK structures on $\mathfrak{g}$ as the structures 
corresponding to left-invariant LCK structures on $G$. 

Let $J \in \End(\mathfrak{g})$ be a complex structure on $\mathfrak{g}$, 
that is, 
$J^2=-\id_{\fg}$ holds. 
We call the complex structure $J$ \emph{integrable} if 
\emph{the Nijenhuis tensor} 
$N_J \colon \mathfrak{g} \otimes \mathfrak{g} \to \mathfrak{g}$ 
defined by 
\[
N_J(X,Y) \coloneqq [X,Y]-[JX,JY]+J[X,JY]+J[JX,Y]
\]
is zero. 
Let $\met{\cdot}{\cdot}$ be an inner product on $\mathfrak{g}$ such that 
$\met{J\cdot}{J\cdot}=\met{\cdot}{\cdot}$ holds. 
We define the fundamental form 
$\omega(\cdot \, , \cdot) \coloneqq \met{J\cdot}{\cdot} \in \bigwedge^2 \mathfrak{g}^*$. 
We call the pair $(\mathfrak{g},J,\met{\cdot}{\cdot})$ LCK 
with the Lee form $\theta \in \mathfrak{g}^*$ if 
$J$ is integrable, 
$d\theta=0$ and $d\omega=\theta\w \omega$ holds, 
where $d$ is the exterior derivative of 
\emph{the Chevalley-Eilenberg complex} $(\bigwedge^{\bullet}\mathfrak{g}^*,d)$. 
We note that 
$d\theta=0$ is equivalent to $\theta|_{[\mathfrak{g},\mathfrak{g}]}=0$ and 
$d\omega=\theta \w \omega$ is equivalent to 
\[
-\omega([X,Y],Z)-\omega([Y,Z],X)-\omega([Z,X],Y)=
\theta(X)\omega(Y,Z)+\theta(Y)\omega(Z,X)+\theta(Z)\omega(X,Y)  
\]
for all $X,Y,Z\in \mathfrak{g}$. 
We call an LCK Lie algebra $(\mathfrak{g},J,\met{\cdot}{\cdot})$ with 
the Lee form $\theta$ \emph{Vaisman} if 
$\ad_{A} \in \End{\mathfrak{g}}$ is skew-symmetric with respect to 
$\met{\cdot}{\cdot}$, 
where $A \in \mathfrak{g}$ is the \emph{Lee vector} determined by 
the relation $\theta(\cdot)=\met{A}{\cdot}$. 

Left-invariant LCK (Vaisman) structures on a Lie group correspond to 
LCK (Vaisman) structures on its Lie algebra. 
Let $G$ be a connected solvable Lie group 
and $\Gamma$ be a lattice in $G$. 
A left-invariant LCK structure on $G$ naturally defines an LCK structure 
on $\Gamma \backslash G$. 
The LCK structure on $\Gamma \backslash G$ obtained in this way 
is also called left-invariant. 

\begin{example}
[\cite{Thu76}, \cite{CFdL86}]
\label{Ex:Heisenberg}
    Let $\mathfrak{g}=\R \times \mathfrak{h}_{2n+1}$ be 
    a nilpotent Lie algebra, 
    where $\mathfrak{h}_{2n+1}$ is the $(2n+1)$-dimensional 
    Heisenberg Lie algebra. 
    There is a basis $\{x_1,y_1,\ldots,x_n,y_n,z_1,z_2\}$ of 
    $\mathfrak{g}$ such that the Lie bracket is given by the relations 
    $[x_i,y_i]=z_1$ for all $1 \le i \le n$. 
    For the dual basis, 
    the indices are written as superscripts, 
    such as $x^i$. 
    We define a metric $\met{\cdot}{\cdot}_0$ on $\mathfrak{g}$ such that 
    the basis above is orthonormal. 
    Let $J_0$ be the complex structure given by 
    $J_0x_i=y_i$ for all $1\le i \le n$ and $J_0z_1=z_2$. 
    Then $J_0$ is integrable and 
    the fundamental form is 
    \[
    \omega_0 =\sum_{i=1}^n x^i \w y^i +z^1 \w z^2 
    \]
    which satisfies $d\omega_0=\theta_0 \w \omega_0$, 
    where $\theta_0=-z^2$. 
    The Lee vector is $A_0=-z_2$. 
    Since the adjoint representation $\ad_{A_0}$ is trivial, 
    it is in particular skew-symmetric. 
    Hence $(\R \times \mathfrak{h}_{2n+1}, J_0,\met{\cdot}{\cdot}_0)$ is 
    a Vaisman Lie algebra. 
    This Lie algebra admits a lattice. 
    We call the associated nilmanifold the \emph{Kodaira-Thurston} manifold, 
    which is the first example of 
    a compact complex manifold which admits a symplectic structure 
    but cannot admit any K\"{a}hler structures. 
\end{example}

Lastly, 
we note that a metric obtained by replacing an LCK metric by 
its scalar multiple is also LCK. 
Thus, 
when discussing the uniqueness of the metric, 
we always express it up to scalar multiple. 

\section{Oeljeklaus-Toma manifolds}
\label{section:Oeljeklaus-Toma manifolds}

Oeljeklaus-Toma (OT, for short) manifolds,  
first defined in \cite{OT05}, 
are higher dimensional generalizations of Inoue-Bombieri surfaces \cite{Ino74}. 
In this section, 
we carefully review their definition and properties 
since they are essential in this paper. 

\subsection{Construction}

Let $K$ be a number field of degree $n = [ K : \Q]$. 
We suppose that $K$ admits $s$ real embeddings 
$\sigma_1, \ldots ,\sigma_s : K \hookrightarrow \R$, 
and $2t$ complex embeddings 
$\sigma_{s+1}, \ldots ,\sigma_{s+2t} : K \hookrightarrow \C$. 
Since the complex embeddings come in conjugate pairs, 
we can label the indices so that $\sigma_{s+i} = \conj{\sigma_{s+t+i}}$ 
for all $1 \le i \le t$. 
We always consider only the case where $s,t \ge 1$. 

Let $\OO{K}$ denote the ring of algebraic integers of $K$, 
$\unit{K}$ denote the group of units in $\OO{K}$ and 
$\unitp{K}$ denote the group of \emph{totally positive units} defined by 
\[
    \unitp{K} \coloneqq 
    \left\{ u \in \unit{K}\, \mid \, 
    \sigma_i(u)>0 \; \text{for all $1 \le i \le s$} \right\}. 
\]
Let $\Half \coloneqq \{ z \in \C\, \mid \, \Im z > 0 \}$ be 
the upper half-plane. 
Consider the action $\OO{K} \curvearrowright \Half^s \times \C^t$ given by translations: 
\[
    T_a(w_1, \ldots ,w_s,z_1, \dots ,z_t) \coloneqq 
    (w_1 + \sigma_1(a), \ldots ,z_t + \sigma_{s+t}(a))
\]
for $a \in \OO{K}$, 
and the action $\unitp{K} \curvearrowright \Half^s \times \C^t$ given by rotations: 
\[
    R_u(w_1, \ldots ,w_s,z_1, \dots ,z_t) \coloneqq 
    (\sigma_1(u) \cdot w_1, \ldots ,\sigma_{s+t}(u) \cdot z_t) 
\]
for $u \in \unitp{K}$. 
These actions define the action 
$( \unitp{K} \ltimes \OO{K} ) \curvearrowright \Half^s \times \C^t$. 
The rank of the group $\OO{K}$ is $n$ since 
$\OO{K} \otimes_{\Z} \Q \simeq K$ holds. 
The group structure of $\unitp{K}$ is described by Dirichlet's unit theorem as follows: 

\begin{theorem}[Dirichlet's unit theorem]
\label{Thm:Dirichlet's unit theorem}
    The image of the following inclusion
    \[
    \begin{array}{c}
        \log : \unitp{K} \hookrightarrow \R^{s+t} \\
        u \mapsto (\log \sigma_1(u), \ldots ,\log \sigma_s(u),
        \log \abs{\sigma_{s+1}(u)}^2, \ldots ,\log \abs{\sigma_{s+t}(u)}^2)
    \end{array}
    \]
    is a lattice in the hyperplane 
    $H \coloneqq \left\{x \in \R^{s+t}\, \mid \, 
    \sum_{i=1}^{s+t} x_i =0  \right\}$. 
\end{theorem}

In particular, 
the rank of the group $\unitp{K}$ is $s+t-1$. 
To make the action 
$( \unitp{K} \ltimes \OO{K} ) \curvearrowright \Half^s \times \C^t$ discrete, 
we take a subgroup $U \subset \unitp{K}$ which has a lower rank. 

\begin{definition}\label{Def:admissible}
    Let $\pr_{\R^s} : \R^{s+t} \to \R^{s}$ be the projection on the first $s$ coordinates. 
    We call a subgroup $U \subset \unitp{K}$ \emph{admissible} if 
    the rank of $U$ is $s$ and the image
    $\pr_{\R^s}(\log (U))$ is a lattice in $\R^s$. 
\end{definition}

By Dirichlet's unit theorem, 
one can always find an admissible subgroup $U$. 
Oeljeklaus and Toma proved in \cite{OT05} that the action 
$( U \ltimes \OO{K} ) \curvearrowright (\Half^s \times \C^t)$ 
is fixed-point-free, properly discontinuous and co-compact for any 
admissible subgroup $U \subset \unitp{K}$. 
Now we can define OT manifolds. 

\begin{definition}
    The \emph{OT manifold} associated to the number field $K$ and 
    to the admissible subgroup $U \subset \unitp{K}$ is 
    a compact complex manifold constructed by the quotient 
    \[
        X(K,U) \coloneqq (\Half^s \times \C^t) / ( U \ltimes \OO{K} ). 
    \]
    We call this an OT manifold of type $(s,t)$. 
\end{definition}

\begin{remark}
    OT manifolds of type $(1,1)$ are known as Inoue-Bombieri surfaces. 
    OT manifolds originally arose from an attempt to generalize 
    Inoue-Bombieri surfaces to manifolds of higher dimension. 
\end{remark}

\subsection{Solvmanifold structure}
Since $\unitp{K} \ltimes \OO{K}$ is a solvable group, 
it is natural to expect that there is a solvmanifold structure 
on an OT manifold. 
In \cite{Kas13b}, 
OT manifolds are constructed as solvmanifolds. 
We shall construct OT manifolds as solvmanifolds and 
show that these coincide with the original ones. 
We also review the structure of solvable Lie algebras associated to 
OT manifolds. 

\begin{lemma}\label{Lem:definition of phi}
    For any admissible subgroup $U$, 
    there is a Lie group homomorphism $\phi$ 
    such that the following diagram 
    is commutative: 
    \begin{equation}\label{Equ:commutativity of phi}
    \begin{tikzcd}[row sep=large, column sep=large]
        U \arrow[r, hook, "\pr_{\R^s} \circ \log"] \arrow[rd, "R"'] 
        & \R^s \arrow[d, "\phi"] \\
        & \GL(\R^s \oplus \C^t), 
    \end{tikzcd}
    \end{equation}
    where $R$ denotes the rotations 
    $R(u) \coloneqq 
    \mathrm{diag}(\sigma_1(u), \ldots ,\sigma_{s+t}(u))$ 
    for $u \in U$. 
\end{lemma}

\begin{proof}
    It is enough to prove that there is a Lie algebra map 
    $d\phi \colon \R^s \to \End(\R^s \oplus \C^t)$ so that 
    the following diagram commutative: 
    \[
    \begin{tikzcd}[row sep=large, column sep=large]
        U \arrow[r, hook, "\pr_{\R^s} \circ \log"] \arrow[rd, "R"'] 
        & \R^s \arrow[d, "\phi"] \arrow[r, "d\phi"]
        & \End(\R^s \oplus \C^t) \arrow[ld, "\exp"]
        \\
        & \GL(\R^s \oplus \C^t). 
    \end{tikzcd}
    \]
    The restriction of the map $d\phi$ to $\pr_{\R^s}(\log (U))$ is 
    determined by the commutativity 
    (up to the choice of a branch of the logarithm). 
    To construct the map $d\phi$, 
    it suffices to simply extend the above map linearly to $\R^s$. 

    Specifically, 
    we construct it as follows. 
    Fix a basis $(a_i)_{i=1}^s$ of $U$. 
    Then $\pr_{\R^s}(\log(a_i))=(\kappa_{1i}, \ldots ,\kappa_{si})$ 
    forms a $\Z$-basis of the lattice $\pr_{\R^s}(\log (U))$, 
    especially $(\kappa_{ij})_{ij} \in \Mat_{s \times s}(\R)$ is an invertible matrix. 
    Take a matrix $C=(c_{ij})_{ij} \in \Mat_{t \times s}(\C)$ 
    so that 
    \begin{equation}\label{Equ:definition of c_ij}
        \sigma_{s+i}(a_k)=
        \exp \left( \sum_{j=1}^s c_{ij} \kappa_{jk} \right) 
    \end{equation}
    for all $1 \le i \le t$. 
    Note that the matrix $C$ depends on 
    the choice of a branch of the logarithm. 
    We define a linear map $d\phi$ as follows: 
    \[
    d\phi(t_1, \ldots ,t_s) \coloneqq 
    \text{diag} ( t_1, \ldots ,t_s, (\sum_{j=1}^s c_{ij} t_j)_{i=1}^t ). 
    \]
    This linear map makes the above diagram commutative. 
\end{proof}

As mentioned in the proof, 
note that the map $\phi$ is not uniquely determined.
Using the map $\phi$, 
we can construct a meta-abelian Lie group 
$G \coloneqq \R^s \ltimes_{\phi} (\R^s \oplus \C^t)$. 
We call this Lie group (and its Lie algebra) a 
\emph{Lie group (algebra) associated to the OT manifold $X(K,U)$}. 
The following lemma is obvious by the commutativity of the diagram (\ref{Equ:commutativity of phi}): 

\begin{lemma}
    The following injective map is a group homomorphism: 
    \[
    \begin{array}{c}
    \iota : U \ltimes \OO{K} \hookrightarrow G \\
    (u,a) \mapsto (\pr_{\R^s}(\log (U)), \sigma_1(a), \ldots ,\sigma_{s+t}(a)). 
    \end{array}
    \]
\end{lemma}

Now we can prove that the meta-abelian Lie group $G$ and its 
discrete subgroup $\iota(U \ltimes \OO{K})$ reconstruct 
an OT manifold. 
The following proposition immediately follows from 
the commutativity of the diagram (\ref{Equ:commutativity of phi}): 

\begin{proposition}
    The following diffeomorphism 
    \[
    \begin{array}{c}
    \psi : \Half^s \times \C^t \stackrel{\sim}{\longrightarrow} 
    \R^s \ltimes_{\phi} (\R^s \oplus \C^t) \\
    ((\sqrt{-1}x_i+y_i)_{i=1}^s,(z_j)_{j=1}^t) \mapsto 
    ((\log x_i)_{i=1}^s, (y_i)_{i=1}^s, (z_j)_{j=1}^t) 
    \end{array}
    \]
    makes the following diagram commutative: 
    \[
    \begin{tikzcd}[row sep=large, column sep=large]
        \Half^s \times \C^t \arrow[r,  "\psi"] \arrow[d, "L_{(u,a)}"'] 
        & \R^s \ltimes_{\phi} (\R^s \oplus \C^t) \arrow[d, "L_{\iota(u,a)}"] \\
        \Half^s \times \C^t \arrow[r,  "\psi"] 
        & \R^s \ltimes_{\phi} (\R^s \oplus \C^t), 
    \end{tikzcd}
    \]
    for all $(u,a) \in U \times \OO{K}$, 
    where $L_{(u,a)}$ denotes the action of $(u,a)$ in the action 
    ($U \ltimes \OO{K}) \curvearrowright (\Half^s \times \C^t)$ 
    and $L_{\iota(u,a)}$ denotes the left multiplication of $\iota(u,a)$ with 
    respect to the Lie group structure. 
\end{proposition}

By the above proposition, 
the map $\psi$ induces the following diffeomorphism 
\[
X(K,U) = (\Half^s \times \C^t) / (U \ltimes \OO{K}) 
\stackrel{\sim}{\longrightarrow} 
\Gamma \backslash G, 
\]
where $\Gamma \coloneqq \iota(U \ltimes \OO{K})$ is the 
lattice in the meta-abelian Lie group $G$. 

The Lie algebra of $G$ is expressed as follows: 

\begin{proposition}\label{Prop:Lie algebra of OT}
    The Lie algebra of $G$ is a meta-abelian Lie algebra   
    $\mathfrak{g} \coloneqq \R^s \ltimes_{d\phi} (\R^s \oplus \C^t)$ 
    where the map $d\phi : \R^s \to \End(\R^s \oplus \C^t)$ 
    is the following: 
    \begin{equation}\label{Equ:Lie algebra of OT}
    d\phi(t_1, \ldots ,t_s) = 
    \mathrm{diag} (t_1, \ldots ,t_s,(\sum_{j=1}^s c_{ij} t_j)_{i=1}^t), 
    \end{equation}
    where the matrix $C=(c_{ij})_{ij}$ is defined in the proof of Lemma \ref{Lem:definition of phi}. 
\end{proposition}

Since $G$ has a lattice, 
the Lie algebra $\mathfrak{g}$ is unimodular. 
Thus we have 
\begin{equation}\label{Equ:sum of c_ij is -1/2}
\Re c_{1j} +\cdots+\Re c_{tj} =-\frac{1}{2}
\end{equation}
for all $1 \le j \le s$. 

We finally note that the Lie group $G$ satisfies the following 

\begin{lemma}\label{Lem:nilradical of OT}
    The nilradical of the Lie group 
    $G = \R^s \ltimes_{\phi} (\R^s \oplus \C^t)$ 
    is $\R^s \oplus \C^t$. 
\end{lemma}

\begin{proof}
    It suffices to show that the nilradical of the Lie algebra 
    $\mathfrak{g} = \R^s \ltimes_{d\phi} (\R^s \oplus \C^t)$ is $\R^s \oplus \C^t$. 
    For all $X \in \R^s \subset \mathfrak{g}$, 
    $\ad_X$ is nilpotent if and only if $X=0$ by the structure of $d\phi$ 
    as in (\ref{Equ:Lie algebra of OT}). 
    Since the nilpotency of a Lie algebra is equivalent to 
    the nilpotency of adjoint representations of its elements, 
    the proof is complete. 
\end{proof}

\subsection{OT manifolds of simple type}

In \cite{OT05}, 
in which OT manifolds are originally defined, 
certain OT manifolds are called \emph{of simple type}. 
The main theorem of this paper is precisely about 
OT manifolds of simple type. 

\begin{definition}
    An OT manifold $X(K,U)$ is \emph{of simple type} if 
    there does not exist proper intermediate field extension 
    $\Q \subsetneq K' \subsetneq K$
    such that $U \subset \unitp{K'}$. 
\end{definition}

It can easily be seen that the above condition is equivalent to $\Q(U)=K$. 
There are other characterizations using the simplicity of lattices as shown in 
the following

\begin{lemma}
    Let $G$ be a Lie group associated to an OT manifold $X(K,U)$. 
    Then, the following conditions are equivalent: 

    \begin{itemize}
    \setlength{\itemsep}{5pt}
        \item[(i)] The rank of any $U$-invariant 
        subgroup of $\OO{K}$ is either $0$ or $\rank{\OO{K}}$,  
        \item[(ii)] $\Q (U) = K$, 
        \item[(iii)] There exists a unit $u \in U$ such that $\deg(u)=[K:\Q]$, 
        \item[(iv)] The lattice $\iota(U \ltimes \OO{K})$ in $G$ is 
        of simple type 
        (see Definition \ref{Def:lattice of simple type}).
    \end{itemize}
\end{lemma}

\begin{proof}
    (i) $\Rightarrow$ (ii)\, : 
    Assume by contradiction that $\Q(U)=K' \subsetneq K$. 
    If $K' \subsetneq K$, 
    the integer ring $\OO{K'}$ is a subgroup of $\OO{K}$ such that 
    $0 < \rank{\OO{K'}} < \rank{\OO{K}}$. \\
    (ii) $\Rightarrow$ (iii)\, : 
    Since the extension $K/\Q$ is finite and separable, 
    there exists an element $u \in U$ such that $\Q (u) = K$. \\
    (iii) $\Rightarrow$ (iv)\, : 
    Let $\{ 0 \} \subsetneq V \subset \OO{K} \otimes_{\Z} \Q$ be a 
    $U$-invariant subspace. 
    For all $u \in U$, 
    we define a linear operator 
    $L_u \colon \OO{K} \otimes_{\Z} \Q \to \OO{K} \otimes_{\Z} \Q$ 
    which is induced by the multiplication of $u$. 
    The minimal polynomial of the operator $L_u$ is 
    the minimal polynomial of $u$. 
    We assume that the degree of the minimal polynomial of $u$ is 
    $[K:\Q]=\dim_{\Q} (\OO{K} \otimes_{\Z} \Q)$. 
    Since $L_u V \subset V$ holds, 
    the minimal polynomial of $L_u |_{V} \colon V \to V$ divides that of $L_u$. 
    From the irreducibility of the minimal polynomial, 
    we obtain that the minimal polynomial of $L_u |_{V}$ is equal to that of $L_u$. 
    Therefore we have $V=\OO{K} \otimes_{\Z} \Q$. \\
    (iv) $\Rightarrow$ (i)\, : 
    Let $H$ be a $U$-invariant subgroup of $\OO{K}$. 
    Then the $\Q$-vector space $H \otimes_{\Z} \Q \subset \OO{K} \otimes_{\Z} \Q$ is 
    $U$-invariant. 
    The irreducibility of the representation of $U$ 
    on the space $\OO{K} \otimes_{\Z} \Q$ 
    implies that the space $H \otimes_{\Z} \Q$ 
    is $0$ or $\OO{K} \otimes_{\Z} \Q$. 
    Since $\rank{H} = \dim_{\Q}(H \otimes_{\Z} \Q)$, 
    the rank of $H$ is $0$ or $\rank{\OO{K}}$. 
\end{proof}

\begin{remark}
    Let $K$ be a number field which has at least one real embedding and one complex embedding. 
    Then we can always take an admissible subgroup $U \subset \unitp{K}$ such that 
    $X(K,U)$ is of simple type. 
\end{remark}

\subsection{LCK metrics on OT manifolds}
\label{subsection:LCK metrics on OT manifolds}

OT manifolds play an important role in the study of LCK geometry. 

\begin{proposition}
[\cite{OT05}]
\label{Prop:OT of type (s,1) admits LCK}
    For all $s \ge 1$, 
    an OT manifold of type $(s,1)$ admits an LCK metric. 
\end{proposition}

\begin{proof}
    Let $\Psi \colon \Half^s \times \C \to \R_{>0}$ be a function defined by 
    \[
    \Psi(w_1, \ldots,w_s,z) \coloneqq (\Im w_1) \cdots (\Im w_s). 
    \]
    We set a strictly plurisubharmonic function 
    $F \coloneqq \Psi^{-1} + \abs{z}^2$ on $\Half^s \times \C$. 
    Then $\omega \coloneqq \sqrt{-1}\pa\bpa F$ is a K\"{a}hler form, 
    which defines a K\"{a}hler metric $g$ on $\Half^s \times \C$. 
    It is easy to check that $\Psi g$ is $U \ltimes \OO{K}$-invariant. 
    Therefore the metric $\Psi g$ defines an LCK metric 
    on $(\Half^s \times \C) / (U \ltimes \OO{K})$. 
\end{proof}

By the construction, 
these LCK metrics are left-invariant with respect to the 
solvmanifold structures on OT manifolds. 
It is shown that 
no LCK metric on OT manifolds can be Vaisman. 

\begin{proposition}[\cite{Kas13b}]
    No OT manifold admits a Vaisman metric. 
\end{proposition}

It is natural to ask the converse assertion to Proposition 
\ref{Prop:OT of type (s,1) admits LCK}, 
that is, 
\emph{is an OT manifold admitting an LCK metric always of type $(s,1)$?} 
The answer to this question is ``yes'' 
as shown in the following

\begin{theorem}[\cite{DV23}]\label{Thm:type of OT with LCK is (s,1)}
    An OT manifold $X(K,U)$ of type $(s,t)$ 
    admits an LCK metric if and only if $t=1$. 
\end{theorem}

This question seems to be considered since OT manifolds were first defined. 
In fact, 
it is proved in \cite{OT05} that OT manifolds of type $(1,t)$ does not 
admit any LCK metrics for all $t>1$. 
The proof is established for several cases in 
\cite{OT05}, \cite{Vul14}, and \cite{Dub14}, 
then completed for remaining cases in \cite{DV23}. 
However, 
with a slight modification of the proof in \cite{DV23}, 
we can provide a concise proof for all cases. 
The proof is included in Appendix \ref{Appendix}
as it is not directly related to the main discussion. 

Here, 
we investigate the relationship between 
the existence of LCK metrics on OT manifolds and 
their Lie algebra structures (\ref{Equ:Lie algebra of OT}). 
The following proposition 
reduces the existence of LCK metrics
to a purely arithmetical condition. 
We note that this proposition is fundamental for the proof 
of Theorem \ref{Thm:type of OT with LCK is (s,1)}. 

\begin{proposition}[\cite{Vul14}, \cite{Dub14}]
\label{Prop:LCK then absolute values of sigma unit is equal}
    Let $X(K,U)$ be an OT manifold of type $(s,t)$. 
    Then, 
    $X(K,U)$ admits an LCK metric if and only if 
    for all $u \in U$, 
    we have 
    \[
    \abs{\sigma_{s+1}(u)}= \cdots = \abs{\sigma_{s+t}(u)}. 
    \]
\end{proposition}

The following proposition shows that 
we can determine the existence of LCK metrics 
from the Lie algebra structures. 

\begin{proposition}
    Let $X(K,U)$ be an OT manifold and $(a_i)_{i=1}^s$ be 
    a basis of $U$. 
    Define a matrix
    $C=(c_{ij})_{ij} \in \Mat_{t \times s}(\C)$ 
    so that it satisfies (\ref{Equ:definition of c_ij}). 
    Then $X(K,U)$ admits an LCK metric if and only if 
    $C$ satisfies 
    \begin{equation}
    \Re c_{ij} =-\frac{1}{2t}
    \end{equation}
    for all $1\le i\le t$ and $1 \le j \le s$. 
\end{proposition}

\begin{proof}
    From Proposition \ref{Prop:LCK then absolute values of sigma unit is equal}, 
    the existence of LCK metrics is equivalent to that 
    $(a_i)_{i=1}^s$ satisfies 
    \[
    \log\abs{\sigma_{s+1}(a_k)}=\cdots=\log\abs{\sigma_{s+t}(a_k)}
    \]
    for all $1\le k \le s$. 
    By the definition of $C$, 
    this is equivalent to $C$ satisfying the equation 
    \[
    \sum_{j=1}^s (\Re{c_{1j}}) \kappa_{jk} = \cdots 
    =\sum_{j=1}^s (\Re{c_{tj}}) \kappa_{jk}
    \]
    for all $1\le k \le s$, 
    where $(\kappa_{ij})_{ij} \in \Mat_{s \times s}(\R)$ is a invertible 
    matrix defined in the proof of Lemma \ref{Lem:definition of phi}. 
    The property that all components in each column of a matrix are equal 
    is preserved when multiplied on the right by an invertible matrix. 
    Thus we have
    \[
    \Re c_{1j} =\cdots =\Re c_{tj}
    \]
    for all $1\le j \le s$. 
    From the equation (\ref{Equ:sum of c_ij is -1/2}), 
    we have
    \[
    \Re c_{ij}=\frac{1}{t} \sum_{i=1}^t \Re c_{ij}  
    =- \frac{1}{2t}. 
    \]
\end{proof}

The Lie algebra structure of an OT manifold is expressed by 
using a matrix $C \in \Mat_{t\times s}(\C)$, 
which is constructed arithmetically, 
as in Proposition \ref{Prop:Lie algebra of OT}. 
However, 
to define an LCK metric on the Lie algebra, 
$C$ does not need to be constructed arithmetically. 

\begin{definition}\label{Def:OT-like}
    Let $C=(c_{ij})_{ij} \in \Mat_{t\times s}(\C)$ be a complex matrix 
    satisfying
    \[
    \Re c_{1j}+\cdots+\Re c_{tj}=-\frac{1}{2}
    \]
    for all $1\le j \le s$. 
    We define a unimodular meta-abelian Lie algebra
    $\fg _C=\R^s \ltimes_{d\phi_C}(\R^s \oplus\C^t)$ 
    where the map $d\phi_C : \R^s \to \End(\R^s \oplus \C^t)$ 
    is the following: 
    \begin{equation}
    d\phi_C(t_1, \ldots ,t_s) = 
    \mathrm{diag} (t_1, \ldots ,t_s,(\sum_{j=1}^s c_{ij} t_j)_{i=1}^t). 
    \end{equation}
    We call a Lie algebra \emph{OT-like} of type $(s,t)$ if it is 
    isomorphic to $\fg_C$ 
    for some matrix $C \in \Mat_{t\times s}(\C)$. 
    Moreover, 
    we call it \emph{LCK OT-like} if we can take the matrix $C$ so that 
    \[
    \Re{c_{ij}}=-\frac{1}{2t}
    \]
    for all $1\le i\le t$ and $1 \le j \le s$. 
    We also call a meta-abelian Lie group 
    $G_C=\R^s \ltimes_{\phi_C}(\R^s \oplus\C^t)$ 
    OT-like (LCK OT-like) of type $(s,t)$, 
    where $\phi_C=\exp(d\phi_C)$. 
\end{definition}

Now we define LCK metrics on LCK OT-like Lie algebras. 
Note that we are also considering cases other than of type $(s,1)$. 

\begin{proposition}
[\cite{Kas13b}]
\label{Prop:LCK metric on LCK OT-like algebras}
    Let $C=(c_{ij})_{ij} \in \Mat_{t\times s}(\C)$ be a complex matrix 
    satisfying 
    \[
    \Re{c_{ij}}=-\frac{1}{2t}
    \]
    for all $1\le i\le t$ and $1 \le j \le s$. 
    Then $\mathfrak{g}_C$ 
    admits an LCK structure $(\mathfrak{g}_C,J_C,\met{\cdot}{\cdot}_C)$ 
    defined by the followings. 
    Denote the standard basis of $\mathfrak{g}_C$ by 
    $u_1, \ldots,u_s \in \R^s$ and 
    $v_1,\ldots,v_s,(x_1+\sqrt{-1}y_1),\ldots ,(x_t+\sqrt{-1}y_t) \in \R^s \oplus \C^t$. 
    For the dual basis, 
    the indices are written as superscripts, 
    such as $u^i$. 
    
    \begin{itemize}
        \item $J_Cu_i=v_i$ and $J_Cx_j=y_j$ 
        for all $1 \le i \le s$ and $1 \le j \le t$. 
        \item The fundamental $2$-form $\omega_C$, 
        which defines the inner product $\met{\cdot}{\cdot}_C$, 
        is defined as follows: 
        \[
        \omega_C= t\sum_{i=1}^s u^i \w v^i + \sum_{i,j=1}^s u^i \w v^j 
        + \sum_{i=1}^t x^i \w y^i. 
        \]
        \item The Lee form $\theta_C$ is 
        \[
        \theta_C= \frac{1}{t} (u^1 + \cdots + u^s). 
        \]
    \end{itemize}
\end{proposition}

We call this metric the \emph{standard metric} on $\mathfrak{g}_C$. 
The corresponding left-invariant LCK metric $g_C$ 
on $G_C$ is also called the standard metric. 

\begin{proof}
    It is easy to check that $\met{\cdot}{\cdot}_C$ is 
    $J_C$-invariant and positive definite, 
    $J_C$ is integrable and 
    the Lee form $\theta_C$ is closed. 
    We will show that $d\omega_C=\theta_C \w \omega_C$. 
    We define $z_i=x_i+\sqrt{-1}y_i$ and $z^i=\frac{1}{2}(x^i-\sqrt{-1}y^i)$. 
    Then, 
    the derivatives for the basis are summarized as follows: 
    \[
    du^i=0, \quad
    dv^i=-u^i \w v^i, \quad
    dz^i= \left( -\sum_{j=1}^s c_{ij} u^j \right) \w z^i. 
    \]
    The derivative $d\omega_C$ is calculated as follows: 
    \begin{align*}
        d\omega_C
        &= -t\sum_{i=1}^s u^i \w dv^i - \sum_{i,j=1}^s u^i \w dv^j 
        + d \left( \sum_{i=1}^t x^i \w y^i \right)  \\
        &= -t\sum_{i=1}^s u^i \w (-u^i \w v^i) - \sum_{i,j=1}^s u^i \w (-u^j \w v^j) 
        + d \left( -2\sqrt{-1}\sum_{i=1}^t z^i \w \conj{z}^i \right) \\
        &= 0 + \theta_C \w \left( t\sum_{i=1}^s u^i \w v^i \right) 
        - 2\sqrt{-1} \sum_{i=1}^t 
        \left( -2\sum_{j=1}^s (\Re c_{ij})u^j \right) \w z^i \w \conj{z}^i \\
        &= \theta_C \w \left( t\sum_{i=1}^s u^i \w v^i \right) + 
        \theta_C \w \left( \sum_{i=1}^t x^i \w y^i \right). 
    \end{align*} 
    Since we have 
    \[
        \theta_C \w \left( \sum_{i,j=1}^s u^i \w v^j \right) 
        =\frac{1}{t} \sum_{i,j,k=1}^s u^k \w u^i \w v^j =0, 
    \]
    it is proved that $d\omega_C=\theta_C \w \omega_C$. 
\end{proof}

\section{Construction of OT manifolds from Lie groups
and their lattices}
\label{section:Construction of OT manifolds from Lie groups
and their lattices}

In the previous section, 
we constructed OT manifolds and Lie groups associated with them. 
The construction of OT manifolds can be seen 
as providing a method to construct a matrix $C \in \Mat_{t \times s}(\C)$ 
such that an OT-like Lie group $G_C$ admits a lattice. 
In this section, 
we consider the conditions on $C$ 
under which a Lie group $G_C$ admits a lattice. 

As shown in Lemma \ref{Lem:nilradical of OT}, 
the nilradical of an OT-like Lie group
$\R^s \ltimes_{\phi_C} (\R^s \oplus \C^t)$ is 
$\R^s \oplus \C^t$. 
So we can consider the simplicity of lattices 
(Definition \ref{Def:lattice of simple type}) on OT-like 
Lie groups. 
The next theorem shows that we can conduct 
the ``converse'' construction provided that 
we only consider lattices of simple type. 

\begin{theorem}
\label{Thm:OT-like Lie group with simple lattice is constructed by OT}
    Let $G$ be an OT-like Lie group of type $(s,t)$. 
    If $G$ admits a lattice of simple type, 
    $G$ is associated to some OT manifold $X(K,U)$ 
    of type $(s,t)$ and of simple type. 
\end{theorem}

We may assume that 
$G=G_C=\R^s \ltimes_{\phi_C} (\R^s \oplus \C^t)$ 
for some matrix $C \in \Mat_{t\times s}(\C)$. 
Firstly, 
we simplify the condition that $G$ has a lattice 
of simple type. 
According to Proposition 
\ref{Prop:meta-abelian with a lattice admits splitting a lattice}, 
we can take a lattice $\Gamma \subset \R^s$ and $P \in \GL(n,\R)$ such that 
\[
P (\phi_C(x))P^{-1} \in \SL(n,\Z), \quad \text{for all} \;x \in \Gamma, 
\]
where $n=s+2t$. 
Now we define a subgroup of $\SL(n,\Z)$ as follows: 
\[
U \coloneqq \left\{
P (\phi_C(x))P^{-1} \in \SL(n,\Z)\, \mid \, 
x \in \Gamma \right\}. 
\]
The group $U \subset \SL(n,\Z)$ is a free abelian subgroup of rank $s$. 
Moreover, 
since $\phi_C(x)$ is diagonal, 
we can write as 
\[
P^{-1}AP = \text{diag} 
(\mu^1_A, \ldots \mu^s_A,\lambda^1_A, \ldots ,\lambda^t_A)
\]
for $A \in U$, 
where $\mu^i_A>0$ and $\lambda^i_A \in \C$. 
Since the coordinate expression of 
$x \in \R^s$ is $(\log\mu^1_A, \ldots ,\log\mu^s_A)$ where 
$A=P(\phi_C(x))P^{-1}$, 
we can say that 
$\left\{ (\log\mu^1_A, \ldots ,\log\mu^s_A) 
\in \R^s \, \mid \, A \in U \right\}$ 
forms a lattice in $\R^s$. 
The condition that the lattice $\Gamma \ltimes_{\phi_C} (P^{-1}\Z^n)$ 
is of simple type is equivalent to the irreducibility of the 
representation of $U$ on $\Q^n$. 

Now we can define a $\Q$-algebra $K \coloneqq \Q[U] \subset \Mat(n,\Q)$. 

\begin{lemma}\label{Lem:K is reduced}
    $K$ is a reduced ring. 
\end{lemma}

\begin{proof}
    We consider an isomorphism of $\R$-algebras  
    $\tilde{\rho} \colon \Mat(n,\R) \to \Mat(n,\R)$ 
    which sends $A$ to $P^{-1}AP$, 
    and let $\rho$ be the restriction of 
    $\tilde{\rho}$ to $K \otimes_{\Q} \R$. 
    Since all matrices in 
    $K \otimes_{\Q} \R = \R[U]$ are diagonalized by $\rho$, 
    we obtain the following diagram: 
    \[
    \begin{array}{c@{\hspace{1cm}}c}
        \begin{tikzcd}
            \tilde{\rho}\, \colon \Mat(n,\R) \arrow[r,  "\sim"] 
            \arrow[d, phantom, "\cup" description]
            & \Mat(n,\R) \arrow[d, phantom, "\cup" description]  \\
            \rho\, \colon K \otimes_{\Q} \R \arrow[r, hookrightarrow] 
            & \R^s \oplus \C^t, 
        \end{tikzcd}
        &
        A \mapsto P^{-1}AP, 
    \end{array}
    \]
    where we consider $\R^s \oplus \C^t$ as a subalgebra of 
    $\Mat(n,\R)$ through the 
    inclusion $(\mu^1, \ldots ,\mu^s,\lambda^1, \ldots ,\lambda^t) 
    \mapsto \text{diag}(\mu^1, \ldots ,\mu^s,\lambda^1, \ldots ,\lambda^t)$. 
    As $\R^s \oplus \C^t$ is reduced, 
    it follows that the ring $K \otimes_{\Q} \R$ is reduced and so is $K$. 
\end{proof}

We write $\rho(A) = 
(\mu^1_A, \ldots \mu^s_A,\lambda^1_A, \ldots ,\lambda^t_A)$ 
for a matrix $A \in \R[U]$, 
which is compatible with the notation for $A \in U$. 

\begin{lemma}\label{Lem:K is a field}
    $K$ is a field. 
\end{lemma}

\begin{proof}
    $K$ is a finite dimensional $\Q$-vector space since $K \subset \Mat(n,\Q)$, 
    so in particular $K$ is an Artinian ring. 
    Moreover, 
    $K$ is reduced by Lemma \ref{Lem:K is reduced}. 
    Therefore the ring $K$ is isomorphic to a finite product of fields 
    $K_1 \times \cdots \times K_r$, 
    where $K_i$ are fields. 
    We assume by contradiction that $r \ge2$. 
    Let $A_i \in \Mat(n,\Q)$ be the image of $1 \in K_i$ under the inclusion 
    $K_i \hookrightarrow K$ and $V_i \subsetneq \Q^n$ be the kernel of $A_i$. 
    For all $i \neq j$, 
    we have $A_i A_j=0$ by the definition of $A_i$. 
    In particular we obtain $\dim(\im A_i) \le \dim V_j$. 
    Thus we have the following inequality:  
    \[
    \dim V_i + \dim V_j = n - \dim(\im A_i) + \dim V_j \ge n. 
    \]
    Therefore we have $V_i \neq \{ 0 \}$ since $\dim V_j < n$. 

    The non-trivial subspace $V_i \subset \Q^n$ is 
    $K$-invariant since $K$ is commutative. 
    This contradicts the assumption 
    that the representation of $U$ on $\Q^n$ is irreducible. 
\end{proof}

\begin{lemma}\label{Lem:[K:Q]=n}
    $[K:\Q]=n$. 
\end{lemma}

\begin{proof}
    We consider $\Q^n$ as a $K$-vector space via 
    $K \subset \Mat(n,\Q)$. 
    The irreducibility of the representation of $U$ on $\Q^n$ implies that 
    $\dim_{K}\Q^n=1$. 
    Thus we have $\dim_{\Q} K =n$. 
\end{proof}

Before the proof of Theorem 
\ref{Thm:OT-like Lie group with simple lattice is constructed by OT}, 
we prove the following basic algebraic lemma. 

\begin{lemma}\label{Lem:isomorphism of F-algebra}
    Let $F$ be a field and 
    $F_1,\ldots,F_n$ be finite extension fields of $F$. 
    Let $\varphi : F_1 \times \cdots \times F_n \to 
    F_1 \times \cdots \times F_n$
    be an automorphism of $F$-algebras. 
    Then, 
    there exists a permutation $\sigma$ of $\{ 1,2,\ldots,n\}$ 
    and isomorphisms of $F$-algebras 
    $\tau_i : F_{\sigma(i)} \to F_i$ $(i=1,2,\ldots,n)$ such that 
    \[
    \varphi(x_1,\ldots,x_n) = \left( 
    \tau_1(x_{\sigma(1)}), \ldots ,\tau_n(x_{\sigma(n)})
    \right). 
    \]
\end{lemma}

\begin{proof}
    Let $\varphi_i$ denotes the composition 
    \[
    F_1 \times \cdots \times F_n \xrightarrow{\varphi} 
    F_1 \times \cdots \times F_n \xrightarrow{\pr_i} F_i. 
    \]
    Then, 
    the kernel of $\varphi_i$ is a prime ideal of $F_1 \times \cdots \times F_n$, 
    so there is some $\sigma(i) \in \{1,2,\ldots,n\}$ such that
    \[
    \ker \varphi_i = F_1 \times \cdots \times \{ 0 \} \times \cdots F_n, 
    \]
    where $\{ 0 \}$ is at the $\sigma(i)$-th component. 
    In other words, 
    we can write $\varphi_i = \tau_i \circ \pr_{\sigma(i)}$ for 
    some injective homomorphism of $F$-algebras $\tau_i : F_{\sigma(i)} \to F_i$. 
    Since $\varphi$ is an automorphism, 
    it follows that $\sigma$ is bijective and $\tau_i$ are isomorphisms of 
    $F$-algebras. 
\end{proof}

\begin{proof}[Proof of Theorem 
\ref{Thm:OT-like Lie group with simple lattice is constructed by OT}]
    By Lemma \ref{Lem:[K:Q]=n}, 
    it follows that the inclusion 
    $\rho\, \colon K \otimes_{\Q} \R \hookrightarrow \R^s \times \C^t$ 
    is an isomorphism of $\R$-algebras. 
    We suppose that $K$ admits $s'$ real embeddings 
    $\sigma_1, \ldots ,\sigma_{s'} : K \hookrightarrow \R$, 
    and $2t'$ complex embeddings 
    $\sigma_{s'+1}, \ldots ,\sigma_{s'+2t'} : K \hookrightarrow \C$ 
    such that $\sigma_{s'+i} = \conj{\sigma_{s'+t'+i}}$ 
    for all $1 \le i \le t'$. 
    By the Minkowski theory \cite{Neu99}, 
    the map
    \[
    \begin{array}{c}
    \sigma\, \colon K \otimes_{\Q} \R \to \R^{s'} \times \C^{t'} \\
    a \otimes 1 \mapsto (\sigma_1(a), \ldots,\sigma_{s'+t'}(a))
    \end{array}
    \]
    is an isomorphism of $\R$-algebras. 
    Thus we obtain the isomorphism of $\R$-algebras 
    $\R^{s} \times \C^{t} \stackrel{\sim}{\to}  \R^{s'} \times \C^{t'}$ 
    through the maps $\rho$ and $\sigma$. 
    By counting the number of homomorphism of $\R$-algebras 
    from $\R^{s} \times \C^{t}$ and $\R^{s'} \times \C^{t'}$ to $\R$, 
    we have $s=s'$. 
    Since $s+2t=s'+2t'$, 
    we also have $t=t'$. 
    By applying Lemma \ref{Lem:isomorphism of F-algebra} 
    to the automorphism 
    $\R^{s} \times \C^{t} \stackrel{\sim}{\to}  \R^{s'} \times \C^{t'}$, 
    after relabeling the indices, 
    we may assume that 
    $\sigma_i(A) = \mu^i_{A}$ and $\sigma_{s+j}(A) = \lambda^j_{A}$ 
    for all $1 \le i \le s, 1 \le j \le t$ and 
    $A \in K \otimes_{\Q} \R = \R[U]$. 

    Since the characteristic polynomials of 
    $A, A^{-1} \in U \subset \GL(n,\Z)$ are 
    monic and has integer coefficients, 
    it follows that $U \subset \unit{K}$. 
    Moreover, 
    we can see that $U$ is an admissible subgroup of $\unitp{K}$ 
    since $\sigma_i(A)=\mu^i_A>0$ and  
    $\left\{ (\log\mu^1_A, \ldots ,\log\mu^s_A) 
    \in \R^s \, \mid \, A \in U \right\}$ 
    forms a lattice in $\R^s$. 

    Now we shall recall the construction of the Lie group 
    associated to $X(K,U)$. 
    The Lie group structure 
    associated to $X(K,U)$ is determined by the map 
    $\phi' \, \colon \R^s \to \GL(\R^s \oplus \C^t)$ which makes 
    the following diagram commutative: 
        \[
            \begin{tikzcd}[row sep=large, column sep=large]
                U \arrow[r, hook, "pr_{\R^s} \circ \log"] \arrow[rd, "R"'] 
                & \R^s \arrow[d, "\phi'"] \\
                & \GL(\R^s \oplus \C^t), 
            \end{tikzcd}
        \]
    where $R$ is rotation as in Lemma \ref{Lem:definition of phi}. 
    But the following diagram is commutative 
    by the definition of $U$ and $\rho$: 
        \[
            \begin{tikzcd}[row sep=large, column sep=large]
                U \arrow[r, hook, "pr_{\R^s} \circ \log"] 
                \arrow[rd,"\rho|_U"'] 
                & \R^s \arrow[d, "\phi_C"] \\
                & \GL(\R^s \oplus \C^t). 
            \end{tikzcd}
        \]
    Therefore, 
    we can take $\phi'$ as $\phi_C$. 
    Hence $\R^s \ltimes_{\phi_C} (\R^s \oplus \C^t)$ is 
    a Lie group associated to $X(K,U)$. 
\end{proof} 

Fortunately, 
it turns out that the assumption of lattice simplicity 
is not needed for LCK OT-like Lie groups. 

\begin{proposition}\label{Prop:LCK OT-like is simple}
    Any lattice of LCK OT-like Lie group is of simple type. 
\end{proposition}

Before the proof, 
we prove a purely algebraic lemma. 

\begin{lemma}\label{Lem:|lambda|<1 then irreducible}
    Let $A \in \SL(n,\Z)$ be a matrix and 
    $\lambda_1\, \ldots ,\lambda_n$ be their eigenvalues. 
    If their absolute values satisfies 
    \[
    \left\{
    \begin{aligned}
        \abs{\lambda_1} &> 1, \\
        \abs{\lambda_i} &< 1, \quad \text{for all $i>1$}, \\
    \end{aligned}
    \right. 
    \]
    then its characteristic polynomial $P_A(x) \in \Z[x]$ is 
    irreducible over $\Q$. 
\end{lemma}

\begin{proof}
    By Gauss's lemma, 
    it is enough to prove that $P_A(x)$ is irreducible over $\Z$. 
    Assume that there are polynomials 
    $P_1(x), P_2(x) \in \Z[x]$ such that $P_A(x)=P_1(x)P_2(x)$. 
    Since $\lambda_1$ is a single root of $P_A(x)$, 
    we may assume that $P_1(\lambda_1)=0$ and $P_2(\lambda_1) \neq 0$. 
    Then it can be seen that 
    the absolute values of all roots of $P_2(x)$ are less than $1$. 
    However, 
    the constant term of $P_2(x)$ is $\pm 1$. 
    This is only possible when $P_2(x)=\pm 1$, 
    which completes the proof. 
\end{proof}

\begin{proof}[Proof of Proposition \ref{Prop:LCK OT-like is simple}]
    By Proposition 
    \ref{Prop:meta-abelian with a lattice admits splitting a lattice}, 
    there exists a subgroup $U \subset \SL(n,\Z)$ of rank $s$ and 
    $P \in \GL(n,\R)$ such that 
    \[
    P^{-1}AP = \text{diag} (\mu^1_A, \ldots \mu^s_A,\lambda^1_A, \ldots ,\lambda^t_A)
    \]
    for all $A \in U$, 
    and 
    $\{ (\log\mu^1_A, \ldots ,\log\mu^s_A) \in \R^s \, \mid \, A \in U \}$ 
    forms a lattice in $\R^s$, 
    as in the proof of Theorem 
    \ref{Thm:OT-like Lie group with simple lattice is constructed by OT}. 
    From the relations between $\mu^i_A$ and $\lambda^i_A$ such that 
    \[
    \lambda^i_A = \exp \left(\sum_{j=1}^s c_{ij} \mu^j_A \right)
    \]
    and the condition $\Re c_{ij}=-1/2t$, 
    we have 
    \begin{equation}\label{Equ:absolute values of lambda are equal}
    \abs{\lambda^1_A}= \cdots = \abs{\lambda^t_A}. 
    \end{equation}
    As $\left\{ (\log\mu^1_A, \ldots ,\log\mu^s_A) 
    \in \R^s \, \mid \, A \in U \right\}$ 
    forms a lattice in $\R^s$, 
    we can take $A_0 \in U$ so that the following conditions: 
    \[
    \left\{
    \begin{aligned}
        \log{\mu^1_{A_0}} &> 0, \\
        \log{\mu^i_{A_0}} &< 0, \quad \text{for all $i>1$}, \\
        \, \sum_{i=1}^s \log{\mu^i_{A_0}} &>0 
    \end{aligned}
    \right. 
    \]
    holds. 
    Equivalently, 
    we have 
    \[
    \left\{
    \begin{aligned}
        \mu^1_{A_0} &> 1, \\
        \mu^i_{A_0} &< 1, \quad \text{for all $i>1$}, \\
        \, \prod_{i=1}^s \mu^i_{A_0} &>1. 
    \end{aligned}
    \right. 
    \]
    By combining this with $\mu^1_{A_0} \cdots 
    \mu^s_{A_0} \abs{\lambda^1_{A_0}}^2 \cdots \abs{\lambda^t_{A_0}}^2=1$ 
    and (\ref{Equ:absolute values of lambda are equal}), 
    we have $\abs{\mu^i_{A_0}}<1$ for all $1\le i \le s$. 
    From Lemma \ref{Lem:|lambda|<1 then irreducible}, 
    the characteristic polynomial of $A_0$ is irreducible over $\Q$. 
    The simplicity of the lattice is equivalent to 
    the irreducibility of the representation of $U$ on $\Q^n$. 
    If the representation of $U$ on $\Q^n$ is reducible, 
    the characteristic polynomial of $A_0$ is reducible over $\Q$. 
    This is not the case. 
\end{proof}

As a consequence, 
we have the following

\begin{corollary}\label{Cor:LCK OT-like with lattice is OT}
    Let $G$ be an LCK OT-like Lie group of type $(s,t)$. 
    If $G$ admits a lattice, 
    we have $t=1$. 
    Moreover, 
    $G$ is a Lie group associated to some OT manifold $X(K,U)$ 
    of type $(s,1)$ and of simple type. 
\end{corollary}

\begin{remark}
    The assumption of lattice simplicity is necessary. 
    For example, 
    let $X(K_i,U_i)$ $(i=1,2)$ are OT manifolds of type $(s_i,t_i)$. 
    Then, 
    the direct product of the associated Lie groups 
    $(\R^{s_1} \ltimes (\R^{s_1} \oplus \C^{t_1})) \times 
    (\R^{s_2} \ltimes (\R^{s_2} \oplus \C^{t_2})) = 
    (\R^{s_1} \oplus \R^{s_2}) \ltimes ((\R^{s_1} \oplus \R^{s_2}) \oplus 
    (\C^{t_1} \oplus \C^{t_2}))$ is an OT-like Lie group admitting a lattice. 
    When trying to prove Theorem 
    \ref{Thm:OT-like Lie group with simple lattice is constructed by OT} 
    about this group, 
    we find that Lemma \ref{Lem:K is a field} does not hold. 
    Even if we consider an OT-like Lie group for which $K$ can be shown to be a field, 
    we still cannot prove Lemma \ref{Lem:[K:Q]=n} without lattice simplicity. 
    This may suggest that a broader class of manifolds than OT manifolds 
    should be considered. 
\end{remark}

\begin{remark}
    According to Proposition \ref{Prop:LCK OT-like is simple},  
    OT manifolds of type $(s,1)$ are of simple type. 
    As a similar result, 
    it is proved in \cite{ADOS24} that 
    OT manifolds admitting a pluriclosed metric are of simple type.  
\end{remark}

\section{LCK structures of meta-abelian Lie algebras 
$\R^m \ltimes \R^n$}\label{section:meta-abelian}

In this section, 
we study LCK structures of meta-abelian 
Lie algebras $\fg=\R^m \ltimes_{d\phi} \R^n$. 
We always assume that $m,n \ge 1$, 
$\fg$ is unimodular and the Lee form does not vanish, 
since solvmanifolds admitting a K\"{a}hler structure are 
well-studied (see for example \cite{Has06}). 
To avoid confusion, 
we write $\R^m,\R^n$ as $\fm, \fn$, respectively.  
We fix an LCK structure $(\fg,J,\met{\cdot}{\cdot})$ whose 
Lee form is $\theta$. 

This section is divided into two subsections. 
In the first subsection, 
we consider the trivial case, 
in which $(\fg,J,\met{\cdot}{\cdot})$ is found to be Vaisman. 
The main results in this section are contained in the second subsection, 
in which we consider the non-Vaisman case. 

\subsection{Vaisman case} 

Firstly, 
we define abelian complex structures. 
We write $\fd$ for an arbitrary Lie algebra 
since we write $\fg$ for a fixed Lie algebra $\R^m \ltimes_{d\phi}\R^n$. 

\begin{definition}
    A complex structure $J$ on a Lie algebra $\fd$ is called 
    \emph{abelian} if the following relation 
    \[
    [Jx,Jy]=[x,y]
    \]
    holds for all $x,y \in \fd$. 
\end{definition}

Unimodular LCK Lie algebras with abelian complex structures are 
classified in \cite{AO15}. 

\begin{proposition}[\cite{AO15}, Theorem 5.1]
    Let $(\fd,J,\met{\cdot}{\cdot})$ be a 
    unimodular LCK Lie algebra. 
    If $J$ is abelian, 
    $(\fd,J,\met{\cdot}{\cdot})$ is isomorphic to 
    $(\R \times \mathfrak{h}_{2d+1},J_0,\met{\cdot}{\cdot}_0)$, 
    where $2d+2=\dim \fd$, 
    up to scalar multiple. 
    In particular, 
    $(\fd,J,\met{\cdot}{\cdot})$ is Vaisman. 
\end{proposition}

Now we consider the meta-abelian Lie algebra 
$(\fg,J,\met{\cdot}{\cdot})$. 

\begin{lemma}\label{Lem:there on n is not zero then J is abelian}
    If $\theta|_{\fn} \neq 0$, 
    $J$ is abelian. 
\end{lemma}

\begin{proof}
    We can take $x \in \fn$ so that $\theta(x) \neq 0$. 
    For all $y,z \in [\fg,\fg] \subset \fn$, 
    we have 
    \[
    \begin{aligned}
        d\omega(x,y,z) &= 0, \\
        \theta \w \omega (x,y,z) &= \theta(x)\omega(y,z). 
    \end{aligned}
    \]
    So we obtain $\met{y}{Jz}=0$, 
    which implies that $[\fg,\fg] \perp J[\fg,\fg]$. 

    By the integrability of $J$, 
    for all $u,v \in \fg$, 
    we have 
    \[
    \begin{aligned}
    [u,v]-[Ju,Jv] &= N_J(u,v) - J([u,Jv]+[Ju,v]) \\
    &= -J([u,Jv]+[Ju,v]). 
    \end{aligned}
    \]
    Therefore we obtain $[u,v]-[Ju,Jv] \in 
    [\fg,\fg] \cap J[\fg,\fg] = \{ 0 \}$, 
    which completes the proof. 
\end{proof}

\begin{corollary}\label{Cor:classification when theta neq 0}
    Let $(\fg=\R^m \ltimes_{d\phi} \R^n, J, \met{\cdot}{\cdot})$ 
    be a unimodular LCK Lie algebra with the Lee form $\theta$. 
    If $\theta|_{\fn} \neq 0$, 
    $(\fg, J, \met{\cdot}{\cdot})$ is isomorphic to 
    $(\R \times \mathfrak{h}_{2d+1},J_0,\met{\cdot}{\cdot}_0)$, 
    where $2d+2=m+n$, 
    up to scalar multiple. 
\end{corollary}

\begin{remark}
    If we assume that $(\fg, J, \met{\cdot}{\cdot})$ is 
    Vaisman beforehand, 
    we can use the results in \cite{Saw07}. 
    In \cite[Theorem 5.6]{Saw07}, 
    it is proved that 
    a unimodular meta-abelian Vaisman Lie algebra
    with $\theta \neq 0$ is isomorphic to
    $(\R \times \mathfrak{h}_{2d+1},J_0,\met{\cdot}{\cdot}_0)$, 
    up to scalar multiple. 
\end{remark}

\subsection{Non-Vaisman case}

In this subsection, 
we assume that $(\fg,J,\met{\cdot}{\cdot})$ is not Vaisman. 
Under the assumption we have $\theta|_{\fn}=0$
from Corollary \ref{Cor:classification when theta neq 0}. 

To state the main theorem of this section, 
we give some definitions. 
Denote by $\fa$ the maximal $J$-invariant subspace 
in $\fn$, 
that is, 
$\fa=\fn \cap J\fn$. 
We define the space $\fk$ 
as the orthogonal complement of $\fa$ 
in $\fn$. 
Denote by $s, 2t$ the dimension of $\fk,\fa$ 
respectively. 
Let $P \colon \fg \to \fm$ be the  
(not necessarily orthogonal) projection with respect to the decomposition 
$\fg=\fm\ltimes_{d\phi} \fn$. 
The following lemma is immediate consequence of the maximality of 
the space $\fa$.  

\begin{lemma}\label{Lem:PJ is injective}
    $PJ|_{\fk} \colon \fk \to \fm$ is injective. 
\end{lemma}

We have $0\le n-2t=s \le m$ from this lemma. 
In the next theorem, 
we assume that the map in Lemma 
\ref{Lem:PJ is injective} is bijective, 
or equivalently, 
$\fn+J\fn=\fg$. 

\begin{theorem}\label{Thm:classification of LCK meta-abelian}
    Let $\fg=\R^m\ltimes_{d\phi}\R^n$ be a unimodular 
    meta-abelian Lie algebra and 
    $(\fg, J, \met{\cdot}{\cdot})$ be a non-Vaisman LCK structure. 
    If $\fn+J\fn=\fg$, 
    then $\fg$ is an LCK OT-like Lie algebra of type $(s,t)$, 
    where $m=s, n=s+2t$. 
    Moreover, 
    the LCK structure $(\fg, J, \met{\cdot}{\cdot})$ is 
    isomorphic to $(\fg_C,J_C,\met{\cdot}{\cdot}_C)$ for 
    some matrix $C \in \Mat_{t\times s}(\C)$, 
    up to scalar multiple. 
\end{theorem} 

The next lemma shows that 
the assumption $\fn+J\fn=\fg$ 
is not needed when $m=1,2$. 

\begin{lemma}\label{Lem:assumption is not needed when m=1,2}
    Let $\fg=\R^m\ltimes_{d\phi}\R^n$ be a unimodular 
    meta-abelian Lie algebra and 
    $(\fg, J, \met{\cdot}{\cdot})$ be a non-Vaisman LCK structure. 
    If $m=1$ or $m=2$, 
    then we have $\fn+J\fn=\fg$. 
\end{lemma}

\begin{proof}
    We have $\theta|_{\fn}=0$ form Corollary 
    \ref{Cor:classification when theta neq 0}. 
    Notice that the parity of $s$ and $m$ is the same 
    since $n+m$, 
    the dimension of $\fg$, 
    is even. 
    Hence the assertion is trivial when $m=1$. 

    We assume by contradiction that 
    $m=2$ and $s=0$. 
    This is the case where $J\fn=\fn$. 
    Take any $u \in \fm$ such that $\theta(u) \neq 0$
    and $x,y \in \fn$. 
    Then we have 
    \[
    \begin{aligned}
        d\omega(u,x,Jy) &= -\omega([u,x],Jy)-\omega(x,[u,Jy]) \\
        &= -\met{[u,x]}{y}-\met{x}{-J[u,Jy]}, \\
        \theta \w \omega (u,x,Jy) &= \theta(u)\met{x}{y}. 
    \end{aligned}
    \]
    So we obtain  
    \[
    -d\phi(u) - (-J (d\phi(u)) J)^* = \theta(u) \id_{\fn}, 
    \]
    where $(-J (d\phi(u)) J)^*$ denotes the adjoint of 
    $-J (d\phi(u)) J$ with respect to $\met{\cdot}{\cdot}$. 
    By taking the trace of this endomorphism, 
    we have
    \[
    -\tr(d\phi(u)) - \tr((-J(d\phi(u))J)^*) = n \theta(u). 
    \]
    However, 
    $\tr(d\phi(u))=0$ by the unimodularity, 
    and we also have 
    \[
    \tr((-J(d\phi(u))J)^*)=\tr(-J(d\phi(u))J)
    =\tr((d\phi(u))J(-J))=\tr(d\phi(u))=0. 
    \]
    Hence we obtain $0=n\theta(u) \neq0$, 
    which is a contradiction. 
\end{proof}

By combining the results of Corollary 
\ref{Cor:LCK OT-like with lattice is OT}, 
Corollary \ref{Cor:classification when theta neq 0}, 
Theorem \ref{Thm:classification of LCK meta-abelian} and Lemma \ref{Lem:assumption is not needed when m=1,2}, 
we have the following 

\begin{theorem}\label{Thm:classification when m=1,2}
    Let $G=\R^m\ltimes_{\phi}\R^n$ be a unimodular 
    meta-abelian Lie group with $m=1$ or $m=2$. 
    If $G$ admits a left-invariant non-Vaisman LCK metric 
    $(G,J,g)$ 
    and a lattice, 
    then we have $n=m+2$ and 
    $G$ is associated to some OT manifold of type $(m,1)$. 
    Moreover, 
    $(G,J,g)$ is isomorphic to $(G_C,J_C,g_C)$ for 
    some $C \in \Mat_{1\times m}(\C)$. 
\end{theorem}

The rest of this section is devoted to the proof of Theorem
\ref{Thm:classification of LCK meta-abelian}. 
We write $\fg$ as 
$\fm \ltimes_{d\phi} (\fk \oplus \fa)$. 
By the assumption $\fn+J\fn=\fg$, 
which is equivalent to saying that $PJ|_{\fk}$ is bijective, 
we have $\dim \fm=\dim \fk=s$. 
Since we also assume that 
$(\fg,J,\met{\cdot}{\cdot})$ is not Vaisman, 
we have $\theta \neq 0$ and $\theta|_{\fn}=0$. 

\begin{lemma}\label{Lem:commutativity of J and dphi on a}
    For all $u \in \fm$ and $x \in \fa$, 
    we have 
    \[
    (d\phi(u))Jx=J(d\phi(u))x. 
    \]
\end{lemma} 

\begin{proof}
    We can take $v \in \fk$ so that 
    $u=PJv$ since $PJ|_{\fk}$ is bijective. 
    From the integrability of $J$, 
    we have 
    \[
    \begin{aligned}
        0 &= N_J(v,x) \\
        &= [v,x]-[Jv,Jx]+J[v,Jx]+J[Jv,x] \\
        &= 0 - (\ad_{Jv})(Jx) + 0 + J(\ad_{Jv})x \\
        &= -(d\phi(u))Jx + J(d\phi(u))x. 
    \end{aligned}
    \]
    In the last equality we used 
    $d\phi(Pz)=\ad_z$ on $\fn$ for all $z \in \fg$. 
\end{proof}

By the above lemma and the definition of the space $\fa$, 
we have the following 

\begin{corollary}\label{Cor:a is ideal}
    The space $\fa$ is an ideal of $\fg$. 
\end{corollary} 

We can also show the following 

\begin{lemma}\label{Lem:k is ideal}
    The space $\fk$ is an ideal of $\fg$. 
\end{lemma}

\begin{proof}
    For all $u \in \fm$, 
    $v \in \fk$ and $x \in \fa$, 
    we have 
    \[
    \begin{aligned}
        d\omega(u,v,Jx) &= -\omega([u,v],Jx)-\omega(v,[u,Jx]) \\
        &= - \met{[u,v]}{x} - \omega(v,J[u,x]) \\
        &= - \met{[u,v]}{x} - 0. 
    \end{aligned}
    \]
    We used Lemma \ref{Lem:commutativity of J and dphi on a} 
    in the second equality, 
    and Corollary \ref{Cor:a is ideal} and $\fk\perp \fa$ 
    in the last equality. 
    We also have 
    \[
    \theta \w \omega (u,v,Jx) = \theta(u)\met{v}{x} = 0. 
    \]
    Hence we obtain $\met{[u,v]}{x}=0$. 
    By the definition of $\fk$, 
    we have $[u,v] \in \fk$. 
\end{proof}

Therefore, 
the representation $d\phi \colon \fm 
\to \End(\fk \oplus \fa)$ is decomposed into 
$d\phi \colon \fm 
\to \End(\fk) \oplus \End(\fa)$. 
From the $J$-invariance of $\fa$ and 
Lemma \ref{Lem:commutativity of J and dphi on a}, 
we can say that 
$\fm$ acts on $\fa$ as complex linear maps 
with respect to the $\C$-linear structure of $(\fa,J)$.

For all $u \in \fm$ and $x,y \in \fa$, 
we have 
\[
\begin{aligned}
    d\omega (u,x,Jy) &= -\met{d\phi(u)x}{y} - \met{x}{d\phi(u)y}, \\
    \theta \w \omega (u,x,Jy) &= \theta(u)\met{x}{y}. 
\end{aligned}
\]
So $d\phi(u)|_{\fa} + \frac{1}{2}\theta(u)$ acts as a skew-Hermitian operator 
on the Hermitian space $(\fa,J,\met{\cdot}{\cdot})$. 
Since these operators 
$\{d\phi(u)|_{\fa} + \frac{1}{2}\theta(u)\}_{u\in\fm}$
are commutative, 
there exists an orthogonal basis $(x_1,y_1,\ldots,x_t,y_t)$ of $\fa$ 
with $Jx_i=y_i$ for all $1\le i \le t$ 
such that the operators 
$\{d\phi(u)|_{\fa} + \frac{1}{2}\theta(u)\}_{u\in\fm}$ 
have the following matrix representation: 
\begin{equation}\label{Equ:diagonalization of dphi on a}
d\phi(u)|_{\fa}+ \frac{1}{2}\theta(u)= \text{diag}( 
\sqrt{-1}\lambda^1_u, \ldots , \sqrt{-1}\lambda^t_u), 
\end{equation}
where $\lambda^i_u$ is a real number. 
Hence the traces of these operators are 
\[
\tr(d\phi(u)|_{\fa})+t \theta(u)=0. 
\]
By the unimodularity of $\fg$, 
we have 
\begin{equation}\label{Equ:trace of dphi on k}
\tr(d\phi(u)|_{\fk}) = -\tr(d\phi(u)|_{\fa}) = t\theta(u). 
\end{equation}

By the following lemma, 
$\fm$ can be replaced by $\fu \coloneqq J \fk$. 

\begin{lemma}\label{Lem:Jk is abelian}
    $\fu = J \fk \subset \fg$ is an abelian 
    Lie subalgebra. 
\end{lemma}

\begin{proof}
    By the integrability of $J$, 
    for all $v_1,v_2 \in \fk$ we have 
    \[
    0=P(N_J(v_1,v_2))
    =PJ([v_1,Jv_2]+[Jv_1,v_2]). 
    \]
    Since $\fk$ is an ideal and $PJ|_{\fk}$ is injective, 
    we have $[v_1,Jv_2]+[Jv_1,v_2]=0$. 
    Therefore, 
    we obtain 
    \[
    [Jv_1,Jv_2] = - N_J(v_1,v_2) + [v_1,v_2] + [v_1,Jv_2]+[Jv_1,v_2]=0,  
    \]
    which completes the proof. 
\end{proof}

Since $P|_{\fu} \colon \fu \to \fm$ is 
bijective, 
we can replace $\fm$ with $\fu$. 
More precisely, 
we express $\fg$ as 
$\fu \ltimes_{d\phi} (\fk \oplus \fa)$, 
where $d\phi$ is regarded as a map 
$\fu \to \End(\fk) \oplus \End(\fa)$ 
through the isomorphism $P|_{\fu}$. 
Since $\fk \perp\fa$, 
we also have $\fu \perp \fa$. 
Thus we obtain $A \in \fu \oplus \fk$, 
where $A$ denotes the Lee vector. 

We define 
$\fg' \coloneqq \fu \ltimes_{d\phi'} \fk$, 
where $d\phi'(u) \coloneqq d\phi(u)|_{\fk}$ 
for all $u \in \fu$. 
The Lie algebra $\fg'$ is a $J$-invariant subalgebra of $\fg$. 
The Lie algebra $\fg'$ inherits the LCK structure
$(\fg',J,\met{\cdot}{\cdot})$ obtained by restricting the 
LCK structure of $\fg$. 
From the integrability of $J$ and $J\fu=\fk$, 
it follows that $J$ is abelian on $\fg'$. 
The properties of $(\fg',J,\met{\cdot}{\cdot})$ are summarized as follows: 
\begin{itemize}
    \item $\theta|_{\fu} \neq 0, \quad \theta|_{\fk}=0$, 
    \item $J$ is abelian, 
    \item $J\fu=\fk$, 
    \item $\tr(d\phi'(u))=t\theta(u)$. 
\end{itemize}

The last one is obtained by (\ref{Equ:trace of dphi on k}). 
Note that $\fg'$ is not unimodular. 
The Lee vector $A$ is also the Lee vector of 
the LCK Lie algebra $(\fg',J,\met{\cdot}{\cdot})$. 
Consider the operator $\Psi \in \End(\fk)$ given by 
$\Psi x \coloneqq [A,x]$ 
for all $x \in \fk$. 

\begin{lemma}
[cf. \cite{AO15}, Lemma 3.4]
\label{Lem:A is self-adjoint}
    The operator $\Psi$ is self-adjoint. 
\end{lemma}

\begin{proof}
    For all $u_1,u_2 \in \fu$, 
    we have
    \[
    \begin{aligned}
        d\omega(JA,u_1,u_2) 
        &= -\omega([JA,u_1],u_2)-\omega(u_1,[JA,u_2]) \\
        &= \omega([A,Ju_1],u_2)+\omega(u_1,[A,Ju_2]) \\
        &= -\met{\Psi Ju_1}{Ju_2} + \met{Ju_1}{\Psi J u_2},  \\
        \theta \w \omega(JA,u_1,u_2)
        &= \theta(JA) \omega(u_1, u_2) 
        - \theta(u_1)\omega(JA,u_2) + \theta(u_2) \omega(JA,u_1) \\
        &= \met{A}{JA} \omega(u_1,u_2) 
        + \theta(u_1)\theta(u_2) - \theta(u_2) \theta(u_1) \\
        &=0, 
    \end{aligned}
    \]
    which completes the proof. 
\end{proof}

By the above lemma, 
we can decompose the space $\fk$ orthogonally 
into the real eigenspaces of $\Psi$ as follows:  
\[
\fk = \bigoplus_{\lambda \in \R} \fk_{\lambda}. 
\]
We also decompose the space $\fu$ orthogonally as follows:  
\[
\fu = \bigoplus_{\lambda \in \R} \fu_{\lambda}, 
\]
where $\fu_{\lambda} \coloneqq J\fk_{\lambda}$. 
Since $\Psi$ commutes with $d\phi'(u)$ for all $u \in \fu$, 
each space $\fk_{\lambda}$ is an ideal of $\fg'$. 

\begin{lemma}\label{Lem:Psi has only one eigenvalue}
    The operator $\Psi$ has only one eigenvalue. 
\end{lemma}

\begin{proof}
    Since $\theta|_{\fu} \neq 0$, 
    We can take $\lambda_0 \in \R$ and $u \in \fu_{\lambda_0}$ so that 
    $\theta|_{\fu_{\lambda_0}}\neq 0$ and $\theta(u_0) \neq 0$. 
    For all $\lambda_1 \neq \lambda_0$ and $u_1 \in \fu_{\lambda_1}$, 
    since $\fk_{\lambda_0}$ and $\fk_{\lambda_1}$ 
    are ideals and orthogonal, 
    we have
    \[
    \begin{aligned}
        d\omega(u_0,u_1,Ju_1)
        &= \omega([u_0,Ju_1],u_1) + \omega(Ju_0,[u_1,Ju_1]) \\
        &= \met{J[u_1,Ju_0]}{u_1} + \met{Ju_0}{[u_1,Ju_1]} \\
        &= 0, \\
        \theta \w \omega(u_0,u_1,Ju_1) 
        &= \theta(u_0)\omega(u_1,Ju_1) - \theta(u_1)\met{u_0}{u_1} \\
        &= \theta(u_0) \|u_1\|^2. 
    \end{aligned}
    \]
    Therefore, 
    we have $u_1=0$.  
    So it follows that 
    $\fu_{\lambda_1}=\{ 0 \}$ for all $\lambda_1 \neq \lambda_0$, 
    which completes the proof. 
\end{proof}

From the above lemma, 
the operator $\Psi$ acts on $\fk$ 
as scalar multiplication by 
$\lambda_0$. 
Since $\tr(\Psi)=\tr(\ad(A)|_{\fk}) = t \theta(A)$, 
we have 
\begin{equation}\label{Equ:lambda_0}
\lambda_0 = \frac{t}{s}\|A\|^2.  
\end{equation}

\begin{lemma}\label{Lem:A is in u}
    $\fu \perp \fk$, 
    in particular $A \in \fu$. 
\end{lemma}

\begin{proof}
    For all $u_1,u_2 \in \fu$, 
    we have 
    \[
    \begin{aligned}
        d\omega(A,Ju_1,Ju_2)
        &= -\lambda_0 \omega(Ju_1,Ju_2) -\lambda_0 \omega(Ju_1,Ju_2), \\
        \theta \w \omega(A,Ju_1,Ju_2) 
        &= \|A\|^2 \omega(Ju_1,Ju_2). 
    \end{aligned}
    \]
    Hence we obtain 
    \[
    (\|A\|^2+2\lambda_0)\omega(Ju_1,Ju_2)=0. 
    \]
    However, we have 
    \[
    (\|A\|^2+2\lambda_0)= \|A\|^2(1+\frac{2t}{s})>0, 
    \]
    which implies $\omega(Ju_1,Ju_2)=\met{Ju_1}{u_2}=0$, 
    and thus $\fu \perp \fk$ follows. 
    Since $\theta|_{\fk}=0$, 
    we have $A \in \fk^{\perp}=\fu$. 
\end{proof}

\begin{lemma}
    $t>0$, 
    or equivalently $\lambda_0>0$. 
\end{lemma}

\begin{proof}
    If $t=0$, 
    we have $\fg=\fg'$. 
    From the equation (\ref{Equ:lambda_0}) and Lemma \ref{Lem:A is in u}, 
    $A$ is in the center of $\fg$. 
    So we can express the meta-abelian Lie algebra 
    $\fg=\R^s \ltimes \R^s$ as 
    $\fg=\R^{s-1} \ltimes (\R A \oplus \R^s)$. 
    In this case, 
    we have $\theta|_{(\R A \oplus \R^s)} \neq 0$. 
    By Corollary \ref{Cor:classification when theta neq 0}, 
    it follows that 
    $(\fg,J,\met{\cdot}{\cdot})$ is Vaisman, 
    which is not the case. 
\end{proof}

Next we consider the \emph{Killing form} of $\fg'$. 
The Killing form vanishes on $\fk$, 
but as shown in the following lemma, 
it is positive definite on $\fu$. 

\begin{lemma}\label{Lem:Killing is positive}
    The following symmetric bilinear operator 
    \[
    \begin{array}{c}
    \kappa_{\fu} : \fu \times \fu {\longrightarrow} \R\\
    (u_1,u_2) \mapsto 
    \tr(d\phi'(u_1)d\phi'(u_2)) 
    \end{array}
    \]
    is positive definite. 
\end{lemma}

\begin{proof}
    For all $u, u_1, u_2\in \fu$, 
    we have 
    \[
    \begin{aligned}
    d\phi'(u_1)d\phi'(u_2)(Ju) 
    &= [u_1,[u_2,Ju]] \\
    &= [u_1,[u,Ju_2]] \\
    &= [u,[u_1,Ju_2]] \\
    &= [-J[u_1,Ju_2],Ju] \\
    &= d\phi'(-J[u_1,Ju_2])(Ju). 
    \end{aligned}
    \]
    Thus we obtain $d\phi'(u_1)d\phi'(u_2)=d\phi'(-J[u_1,Ju_2])$. 
    By taking the trace, 
    we have
    \begin{equation}\label{Equ:Killing and theta}
    \kappa_{\fu}(u_1,u_2) = \tr(d\phi'(-J[u_1,Ju_2])) = t \theta(-J[u_1,Ju_2]). 
    \end{equation}

    Therefore it is enough to show that 
    \[
    \theta(-J[u,Ju])>0
    \]
    for all $0 \neq u \in \fu$. 
    From the equations 
    \[
    \begin{aligned}
        d\omega(A,u,Ju) 
        &= \omega([A,Ju],u) + \omega(A,[u,Ju]) \\
        &= - \lambda_0 \|u\|^2 + \theta(-J[u,Ju]), \\
        \theta \w \omega(A,u,Ju) 
        &= \theta(A)\omega(u,Ju) - \theta(u)\omega(A,Ju) \\
        &= \|A\|^2 \|u\|^2 - \theta(u)^2, 
    \end{aligned}
    \]
    we have 
    \[
    \theta(-J[u,Ju]) = \lambda_0 \|u\|^2 + 
    (\|A\|^2 \|u\|^2 - \theta(u)^2). 
    \]
    By the Cauchy-Schwarz inequality, 
    we have 
    \[
    \theta(u)^2 = \met{A}{u}^2 \le \|A\|^2 \|u\|^2, 
    \]
    which implies that $\theta(-J[u,Ju])>0$ since $\lambda_0>0$. 
\end{proof}

From the above lemma, 
$\kappa_{\fu}$ defines an inner product on $\fu$. 
We define an inner product $\kappa_{\fk}$ on $\fk$ through 
the isomorphism $J \colon \fu \stackrel{\sim}{\to} \fk$ as follows: 
\[
\kappa_{\fk}(v_1,v_2) \coloneqq \kappa_{\fu}(Jv_1,Jv_2), 
\]
where $v_1,v_2 \in \fk$. 
The following lemma follows essentially because of 
the skew-symmetricity of adjoint operators with respect to the Killing form. 

\begin{lemma}\label{Lem:kappa is self-adjoint}
    For all $u \in \fu$, 
    $d\phi'(u)$ is self-adjoint 
    with respect to the inner product $\kappa_{\fk}$. 
\end{lemma}

\begin{proof}
    For all $u_1,u_2,u_3 \in \fu$ we have 
    \[
    \kappa_{\fk}(d\phi'(u_1)Ju_2, Ju_3) = \kappa(-J[u_1,Ju_2],u_3)
    = \tr(\ad_{(-J[u_1,Ju_2])} \ad_{u_3}). 
    \]
    Since $J$ is abelian, 
    we have $\ad_{u}J = -\ad_{Ju}$ for all $u \in \fu$. 
    Therefore, 
    we obtain 
    \[
    \begin{aligned}
        \ad_{(-J[u_1,Ju_2])} \ad_{u_3} 
        &= \ad_{[u_1,Ju_2]} J \ad_{u_3} \\
        &= (\ad_{u_1} \ad_{Ju_2} - \ad_{Ju_2} \ad_{u_1}) J \ad_{u_3} \\
        &= \ad_{u_1} \ad_{Ju_2} J \ad_{u_3} \\
        &= \ad_{u_1} \ad_{u_2} \ad_{u_3}. 
    \end{aligned}
    \]
    This is symmetric with respect to $u_1$,$u_2$ and $u_3$, 
    which completes the proof. 
\end{proof}

Since the operators $\{d\phi'(u)\}_{u \in \fu}$ are commutative, 
we can take an orthonormal basis $v_1, \ldots ,v_s \in \fk$ with 
respect to $\kappa_{\fk}$ so that $\{d\phi'(u)\}_{u \in \fu}$ 
act on $\fk$ diagonally. 
We also take a basis of $u_1, \ldots , u_s \in \fu$ defined by 
$Ju_i = v_i$ for all $1 \le i \le s$. 
This basis is orthonormal with respect to $\kappa_{\fu}$. 
By multiplying each vector in the basis by $\pm 1$, 
we may assume that $\theta(u_i)\ge0$. 

We define real numbers $a_{ij}$ so that the commutator satisfies 
$[u_i,v_j]=a_{ij}v_j$ for all $1 \le i,j \le s$. 
From the equation (\ref{Equ:Killing and theta}), 
for all $1 \le i,j \le s$ we have 
\[
\delta_{ij} = \kappa_{\fu}(u_i,u_j) = t\theta(-J[u_i,v_j]) 
= t a_{ij} \theta (u_j). 
\]
Thus we obtain $1=t a_{ii} \theta(u_i)$, 
in particular $\theta(u_i) \neq 0$ for all $1 \le i \le s$. 
Hence if $i \neq j$, 
we have $a_{ij}=0$. 
Moreover, 
we have 
\[
t\theta(u_i) = \tr(d\phi'(u_i)) = a_{i1}+\cdots +a_{is} = a_{ii}, 
\]
which implies $a_{ii}^2=1$. 
By the assumption $\theta(u_i) \ge 0$,
we obtain 
\begin{equation}\label{Equ:structure of dphi on k}
    [u_i,v_j]=\delta_{ij} v_j, \quad \theta(u_i)=\frac{1}{t}. 
\end{equation}

Now we can proof the theorem. 

\begin{proof}[Proof of Theorem 
\ref{Thm:classification of LCK meta-abelian}]

\
Fix the basis $u_1,\ldots,u_s,v_1,\ldots,v_s,x_1,y_1,\ldots,x_t,y_t \in \fg$ 
constructed in the argument so far. 
From the equations (\ref{Equ:diagonalization of dphi on a}) and 
(\ref{Equ:structure of dphi on k}), 
we have
\[
    d\phi(t_1, \ldots ,t_s) = 
    \mathrm{diag} (t_1, \ldots ,t_s,(\sum_{j=1}^s c_{ij} t_j)_{i=1}^t)
\]
where $C=(c_{ij})_{ij}\in\Mat_{t\times s}(\C)$ is a matrix 
satisfying $\Re c_{ij} = -\frac{1}{2t}$ for all 
$1\le i \le t$ and $1 \le j \le s$. 
We also have 
\[
\theta=\frac{1}{t} (u^1 + \cdots + u^s). 
\]
Therefore it is shown that $\fg$ is an LCK OT-like Lie algebra of type $(s,t)$ 
where $m=s, n=s+2t$. 

We next show that the inner product $\met{\cdot}{\cdot}$ 
is isomorphic to $\met{\cdot}{\cdot}_C$  
up to scalar multiple. 
We already know that the basis $x_1,y_1,\ldots,x_t,y_t$ of $\fa$ is 
orthonormal and $\fu \oplus \fk \oplus \fa$ 
is the orthogonal decomposition. 

We define real numbers $b_{ij}$ so that 
$\met{u_i}{u_j}=b_{ij}$ for all $1 \le i,j \le s$. 
By using these constants, 
for all $1 \le i,j,k \le s$ we have 
\[
\begin{aligned}
    d\omega(u_i,u_j,v_k) &= -\delta_{ik}b_{kj} + \delta_{jk} b_{ki}, \\
    \theta \w \omega(u_i,u_j,v_k) &= \frac{1}{t}(b_{jk}-b_{ik}). 
\end{aligned}
\]
Therefore we obtain 
\[
(\frac{1}{t}+\delta_{ik})^{-1} b_{ik} 
= (\frac{1}{t}+\delta_{jk})^{-1} b_{jk}. 
\]
From this equation, 
there is a constant $\gamma \in \R$ such that 
\[
(\frac{1}{t}+\delta_{ij})^{-1} b_{ij} = \gamma
\]
holds for all $1 \le i,j \le s$. 
By replacing the inner product $\met{\cdot}{\cdot}$ 
with its constant multiple, 
we may assume that $\gamma=t$. 
Thus, we have 
\begin{equation}\label{Equ:b_ij}
b_{ij} = 1+t\delta_{ij}
\end{equation}
for all $1 \le i,j \le s$. 
Note that we also show that $\met{v_i}{v_j}=1+t\delta_{ij}$. 

From the orthogonality of the spaces $\fu$, $\fk$ and $\fa$, 
along with the equation (\ref{Equ:b_ij}), 
we can express the fundamental form $\omega$ as follows: 
\[
    \omega= t\sum_{i=1}^s u^i \w v^i + \sum_{i,j=1}^s u^i \w v^j 
    + \sum_{i=1}^t x^i \w y^i, 
\]
which completes the proof. 
\end{proof}

\appendix

\section{Proof of Theorem \ref{Thm:type of OT with LCK is (s,1)}}
\label{Appendix}

The new ideas presented in this appendix are 
primarily due to Junnosuke Koizumi. 

In this appendix, 
we prove Theorem \ref{Thm:type of OT with LCK is (s,1)} assuming Proposition \ref{Prop:LCK then absolute values of sigma unit is equal}. 
As explained in subsection \ref{subsection:LCK metrics on OT manifolds}, 
this has already been shown in \cite{DV23}, 
but a concise proof can be given by slightly modifying the argument 
in \cite{DV23}. 
At the end of the appendix, 
we also discuss on the existence of pluriclosed metrics on OT manifolds 
since the similar argument can be applied. 

\begin{theorem}\label{Thm:LCK OT is of type (s,1)}
    Let $K$ be a number field with real embeddings $\sigma_1,\dots,\sigma_s$ 
    and complex embeddings $\sigma_{s+1},\dots,\sigma_{s+2t}$, 
    where $\sigma_{s+i}=\overline{\sigma_{s+t+i}}$. 
    Let $U$ be a subgroup of $\unit{K}$ such that for any $u\in U$, we have
    \[
        \abs{\sigma_{s+1}(u)}=\cdots=\abs{\sigma_{s+t}(u)}.
    \]
    If $s\geq 1$ and $t\geq 2$, then $\rank U\leq s-1$.
\end{theorem}

\begin{proof}
    Define a group homomorphism $L\colon \unit{K} \to \R^{s+t}$ by
    \[
        L(u)=(\log \abs{\sigma_1(u)}, \ldots,\abs{\sigma_{s+t}(u)}).
    \]
    The kernel of $L$ is the group of roots of unity, 
    and the image of $L$ is a lattice in the hyperplane
    \[
        H\colon x_1+\cdots+x_s+2x_{s+1}+\cdots+2x_{s+t}=0
    \]
    by Dirichlet's unit theorem.
    We define hyperplanes $H_1,\dots,H_{t-1}\subset H$ by
    \[
        H_i\colon x_{s+i}=x_{s+i+1}.
    \]
    Then $L(U)$ is contained in the subspace $V=H_1\cap\cdots\cap H_{t-1}$, 
    which has dimension exactly $s$.
    We will prove that $L(U)$ is contained in a union of 
    finitely many hyperplanes of $V$, 
    which shows that $\rank U\leq s-1$.
    More precisely, 
    we define hyperplanes $P_2,\dots,P_{s+t}\subset V$ by
    \[
        P_i\colon x_1=x_i,
    \]
    and for $2\leq i,j,k\leq s+t$, 
    we define a hyperplane $P_{ijk}\subset V$ by
    \[
        P_{ijk}\colon x_1+x_i=x_j+x_k.
    \]
    We will prove that
    \[
        L(U)\subset \left(
        \bigcup_{i=2}^{s+t}P_i \right)
        \cup \left( 
        \bigcup_{i,j,k=2}^{s+t}P_{ijk} \right).
    \]

    Let $u\in U$.
    First we consider the case where $\deg(u)<[K:\Q]$.
    Let $K'=\Q(u)\subset K$, 
    and let $\sigma'$ be the real embedding of 
    $K'$ lying under $\sigma_1$.
    The number of embeddings of $K$ extending 
    $\sigma'$ is $[K:K']\geq 2$, 
    so there must exist $2\leq i\leq s+2t$ 
    such that $\sigma_i$ is an extension of $\sigma'$.
    In this case we get
    \[
        \sigma_1(u)=\sigma'(u)=\sigma_i(u).
    \]
    Therefore, 
    if we set
    \[
        \rho(i)=\begin{cases}
            i&(1\leq i\leq s+t)\\
            i-t&(s+t+1\leq i\leq s+2t),
        \end{cases}
    \]
    then we have $L(u)\in P_{\rho(i)}$.

    Next we consider the case where $\deg(u)=[K:\Q]$.
    In this case, 
    $\sigma_1(u),\dots,\sigma_{s+2t}(u)$ are the distinct conjugates of $u$ over $\Q$.
    Therefore, 
    there exists some $\tau\in \Gal(\overline{\Q}/\Q)$ 
    such that $\tau(\sigma_{s+1}(u))=\sigma_1(u)$.
    Since we are assuming that $t\geq 2$, 
    we have
    \[
        \sigma_{s+1}(u)\sigma_{s+t+1}(u)=\sigma_{s+2}(u)\sigma_{s+t+2}(u).
    \]
    Applying $\tau$ to this equation yields a new relation
    \[
        \sigma_1(u)\sigma_i(u)=\sigma_j(u)\sigma_k(u)
    \]
    for some $2\leq i,j,k\leq s+2t$.
    This shows that $L(u)\in P_{\rho(i)\rho(j)\rho(k)}$ and the proof is complete.
\end{proof}

We obtain Theorem \ref{Thm:type of OT with LCK is (s,1)} as a corollary of the above theorem. 

It is worth noting, 
even if unrelated to the central purpose of this paper, 
that a similar argument can be applied to the case of pluriclosed metrics as well. 
A \emph{pluriclosed metric} (also known as \emph{SKT metric}) 
is a metric $g$ on 
complex manifold such that $\pa \bpa \omega=0$, 
where $\omega$ is the fundamental form associated to $g$. 
As in the case of LCK metrics, 
there is a purely arithmetical condition 
for the existence of pluriclosed metrics on OT manifolds.

\begin{proposition}[\cite{Oti22}]
    Let $X(K,U)$ be an OT manifold of type $(s,t)$. 
    $X(K,U)$ admits a pluriclosed metric if and only if 
    $s \le t$ and after relabeling the indices of $\sigma_1, \ldots, \sigma_s$, 
    for any $u \in U$, 
    we have 
    \[
        \begin{cases}
            \sigma_1(u)|\sigma_{s+1}(u)|^2=\cdots=\sigma_s(u)|\sigma_{2s}(u)|^2=1,\\
            |\sigma_{2s+1}(u)|=\cdots=|\sigma_{s+t}(u)|=1.
        \end{cases}
    \]
\end{proposition}

In \cite{ADOS24}, 
it is proved that conditions in the above proposition cannot hold 
unless $s=t$. 
Here, 
we present an alternative proof by employing 
the same technique in the proof of Theorem 
\ref{Thm:LCK OT is of type (s,1)}. 

\begin{theorem}
    Let $K$ be a number field with real embeddings $\sigma_1,\dots,\sigma_s$ 
    and complex embeddings $\sigma_{s+1},\dots,\sigma_{s+2t}$, 
    where $\sigma_{s+i}=\overline{\sigma_{s+t+i}}$ and $s\leq t$.
    Let $U$ be a subgroup of $\unit{K}$ such that for any $u\in U$, we have
    \[
        \begin{cases}
            \sigma_1(u) \abs{\sigma_{s+1}(u)}^2=\cdots=
            \sigma_s(u) \abs{\sigma_{2s}(u)}^2=1,\\
            \abs{\sigma_{2s+1}(u)}=\cdots=\abs{\sigma_{s+t}(u)}=1.
        \end{cases}
    \]
    If $1\leq s<t$, then $\rank U\leq s-1$.
\end{theorem}

\begin{proof}
    Define a map $L\colon \unit{K} \to \mathbb{R}^{s+t}$ 
    and a hyperplane $H\subset \mathbb{R}^{s+t}$ 
    as in the proof of Theorem \ref{Thm:LCK OT is of type (s,1)}.
    We define hyperplanes $H_1,\dots,H_t\subset H$ by
    \[
        H_i\colon
        \begin{cases}
            x_i+2x_{s+i}=0\quad(1\leq i \leq s),\\
            x_{s+i}=0\quad(s+1\leq i \leq t).
        \end{cases}
    \]
    Then $L(U)$ is contained in the subspace $V=H_1\cap\cdots\cap H_t$, 
    which has dimension exactly $s$.
    We will prove that $L(U)$ is contained in a union of finitely many hyperplanes of $V$, 
    which shows that $\rank U\leq s-1$.
    More precisely, 
    we define hyperplanes $P_2,\dots,P_{s+t}\subset V$ and $Q_2,\dots,Q_{s+t}\subset V$ by
    \[
        P_i\colon x_1=x_i,\quad Q_i\colon x_1+x_i=0.
    \]
    We will prove that
    \[
        L(U)\subset \left( \bigcup_{i=2}^{s+t}P_i \right) 
        \cup \left( \bigcup_{i=2}^{s+t}Q_i \right).
    \]

    Let $u\in U$.
    If $\deg(u)<[K:\Q]$, 
    then we have $L(u)\in P_{\rho(i)}$ for some $2\leq i\leq s+2t$ 
    as in the proof of Theorem \ref{Thm:LCK OT is of type (s,1)}.
    Suppose that $\deg(u)=[K:\Q]$.
    In this case, 
    $\sigma_1(u),\dots,\sigma_{s+2t}(u)$ are the distinct conjugates of $u$ over $\Q$.
    Therefore, 
    there exists some $\tau\in \Gal(\overline{\Q}/\Q)$ 
    such that $\tau(\sigma_{2s+1}(u))=\sigma_1(u)$.
    Since we are assuming that $s<t$, 
    we have
    \[
        \sigma_{2s+1}(u)\sigma_{2s+t+1}(u)=1.
    \]
    Applying $\tau$ to this equation yields a new relation
    \[
        \sigma_1(u)\sigma_i(u)=1
    \]
    for some $2\leq i\leq s+2t$.
    This shows that $L(u)\in Q_{\rho(i)}$ and the proof is complete.
\end{proof}

\begin{corollary}[\cite{ADOS24}]
    An OT manifold $X(K,U)$ of type $(s,t)$ 
    admits a pluriclosed metric if and only if $s=t$ and  
    after relabeling the indices of $\sigma_1, \ldots, \sigma_s$, 
    for any $u \in U$, 
    we have 
    \[
        \sigma_1(u)\abs{\sigma_{s+1}(u)}^2=\cdots=
        \sigma_s(u)\abs{\sigma_{s+t}(u)}^2=1. 
    \]
\end{corollary}

\begin{remark}
    Unlike the case in LCK metrics, 
    it cannot be determined whether an OT manifold admits a pluriclosed 
    metric based only on its type $(s,t)$. 
    For each $s \ge 1$, 
    the construction of an OT manifold of type $(s,s)$ which admits 
    a pluriclosed metric is given in \cite{Dub21}. 
\end{remark}

\bibliographystyle{amsalpha}
\bibliography{OT}
\end{document}